\documentclass[11pt]{article}
\usepackage{graphicx}
\usepackage{amsmath, amsthm, amssymb}
\usepackage{mathrsfs,times}
\title{Geometric ergodicity of a bead-spring pair with stochastic Stokes forcing}
\author{Jonathan C. Mattingly\thanks{Department of Mathematics, Duke University, Durham, NC.}, Scott A. McKinley{\thanks{Department of Mathematics, University of Florida, Gainesville, FL.}}, Natesh S. Pillai{\thanks{Department of Statistics, Harvard University, Cambridge, MA.}}}

\begin{document}

\newcounter{assume}
\newcounter{other}
\newcounter{condition}

\newtheorem{theorem}{Theorem}[section]
\newtheorem{prop}[theorem]{Proposition}
\newtheorem{defn}[theorem]{Definition}
\newtheorem{corollary}[theorem]{Corollary}
\newtheorem{lemma}[theorem]{Lemma}
\newtheorem{fact}[theorem]{Fact}

\theoremstyle{plain}
\newtheorem{cond}[condition]{Condition}

\theoremstyle{remark}
\newtheorem{remark}[theorem]{Remark}

\theoremstyle{plain}
\newtheorem{assumption}[assume]{Assumption}

\setlength{\unitlength}{1 cm}

\def\drag{\gamma}
\def\viscosity{\nu}

\def\Av{\mathbf{A}}
\def\cv{\mathbf{c}}
\def\ev{\mathbf{e}}
\def\fv{\mathbf{f}}
\def\Fv{\mathbf{F}}
\def\gv{\mathbf{g}}
\def\hv{H}
\def\jv{j}
\def\kv{k}
\def\kvperp{k^\perp}
\def\Kv{\mathbf{K}}
\def\mv{M}
\def\Ov{\mathbf{O}}
\def\rv{R}
\def\rvdot{\dot{\mathbf{r}}}
\def\uv{u}
\def\Uv{\mathbf{U}}
\def\vv{\mathbf{v}}
\def\Vv{\mathbf{V}}
\def\xv{X}
\def\Xv{X}
\def\yv{Y}
\def\zv{Z}
\def\wv{\mathbf{w}}
\def\Wv{W}
\def\rhov{\boldsymbol{\rho}}
\def\fluidv{u}
\def\fluidvhat{\hat u}
\def\noisev{\mathbf{F}}
\def\brownv{\mathbf{B}}
\def\ouvcos{\mathbf{y}}
\def\ouvsin{\mathbf{z}}
\def\zerov{0}
\def\xiv{\boldsymbol{\xi}}
\def\zetav{\boldsymbol{\zeta}}
\def\length{\lambda}
\def\rgood{\mathcal{G}_r}
\def\zgood{\mathcal{G}_z}
\def\Lambdav{\Lambda}
\def\kmin{\hat k}

\def\rhomax{\rho_{max}}
\def\rhogood{\rho_+}
\def\Phifene{\Phi_{\text{FENE}}}

\def\imat{\mathbf{I}}

\def\abb{\mathbb{A}}
\def\kbb{\mathbb{K}}
\def\lbb{\mathbb{L}}
\def\pbb{\mathbb{P}}
\def\rbb{\mathbb{R}}
\def\tbb{\mathbb{T}}
\def\xbb{\mathbb{X}}
\def\zbb{\mathbb{Z}}
\def\torus{\mathbb{T}}
\def\modeset{\mathcal{K}}
\def\ocal{\mathcal{O}}
\def\dcal{\mathcal{D}}

\def\f{\varphi}
\def\Ep{\mathcal{E}}
\def\ep{\epsilon}
\def\gen{\mathcal{L}}
\def\indicator{1}
\def\kvsq{|\kv|^2}
\def\projop{\Big(I - \frac{\kv \otimes \kv}{\kvsq}\Big)}
\def\ddt{\frac{d}{dt}}
\def\cpt{\mathcal{C}}
\def\moll{\psi}

\def\pcal{\mathcal{P}}
\def\rcal{\mathcal{R}}
\def\good{\mathcal{G}}
\def\bad{\mathcal{B}}
\def\tkernel{\mathcal{P}}
\def\state{\mathcal{X}}
\def\statemeas{\mathcal{E}}

\def\factname{Fact }

\newcommand{\vect}[2]{\binom{#1}{#2}}

\newcommand{\E}[1]{\mathbb{E}{\left[ #1\right]}}
\newcommand{\Ex}[2]{\mathbb{E}_{#1}{\left[ #2\right]}}
\newcommand{\p}[1]{\mathbb{P}{\left\{ #1\right\}}}
\newcommand{\px}[2]{\mathbb{P}_{#1}{\left\{ #2\right\}}}
\newcommand{\borel}[1]{\mathscr{B}(#1)}
\newcommand{\cntr}{\cpt}
\newcommand{\ip}[1]{{\langle#1\rangle}}
\renewcommand{\P}{\mathbf{P}}

\newcommand{\Norm}[1]{|||#1|||}
\newcommand{\norm}[1]{\|#1\|}
\newcommand{\cP}{\mathcal{P}}
\newcommand{\cR}{\mathcal{R}}
\newcommand{\cB}{\mathcal{B}}
\newcommand{\cC}{\mathcal{C}}
\newcommand{\cS}{\mathcal{S}}
\newcommand{\cI}{\mathcal{I}}
\newcommand{\cF}{\mathcal{F}}
\newcommand{\cK}{\mathcal{K}}
\newcommand{\cE}{\mathcal{E}}
\newcommand{\Cm}{\cH}
\newcommand{\one}{\mathbf{1}}
\newcommand{\cH}{\mathcal{H}}
\newcommand{\Z}{\mathbf{Z}}
\newcommand{\dd}{\rho}
\renewcommand{\d}{d}
\newcommand{\rhoplus}{\rho_{+}'}

\newcommand{\T}{\mathbf{T}}
\newcommand{\N}{\mathbf{N}}
\newcommand{\R}{\mathbf{R}}
\newcommand{\X}{\mathbf{X}}
\newcommand{\ccdot}{\;\cdot\;}
\newcommand{\eqdef}{\stackrel{\mbox{\tiny def}}{=}}

\renewcommand{\theenumi}{\roman{enumi}}
\renewcommand{\theenumii}{\roman{enumii}}
\renewcommand{\theenumiii}{\arabic{enumiii}}
\renewcommand{\labelenumi}{(\theenumi)}

\maketitle

\begin{abstract}
We consider a simple model for the fluctuating hydrodynamics of a
flexible polymer in dilute solution, demonstrating geometric
ergodicity for a pair of particles that interact with each other
through a nonlinear spring potential while being advected by a
stochastic Stokes fluid velocity field.  This is a generalization of
previous models which have used linear spring forces as well as
white-in-time fluid velocity fields.

We follow previous work combining control theoretic arguments,
Lyapunov functions, and hypo-elliptic diffusion theory to prove
exponential convergence via a Harris chain argument. In addition we
allow the possibility of excluding certain ``bad'' sets in phase space in
which the assumptions are violated but from which the system leaves
with a controllable probability. This allows for the treatment of
singular drifts, such as those derived from the Lennard-Jones
potential, which is a novel feature of this work.
\end{abstract}

\section{Introduction}

The study of polymer stretching in random fluids has been identified
as a first step in the much larger project of modeling and
understanding drag reduction in polymer solutions
\cite{2000-chertkov-polymer-stretching} and theoretical focus has
been brought on the dynamics of simple dumbbell models
\cite{2002-delucia-dummbbell-laminar-flow}, \cite{2005-celani},
\cite{2005-afonso}.  Of particular interest is the experimentally
observed phenomenon called the coiled state / stretched state phase
transition \cite{2005-gerashchenko}.  Mathematically this transition
has been characterized by seeking models which admit solutions that
are ergodic for only certain regions of parameter space \cite{2005-celani}. In this paper we
address the topic of how to prove ergodicity for a wide range of
models that generalize preceding work.

Let $\xv_1(t)$ and $\xv_2(t)$ denote the respective positions in
$\rbb^2$ of two polymer ``beads'' connected by a ``spring'' at time
$t$. Depending on the scale of interest, these beads may be thought of as consecutive segments (consisting of something like 50 monomers) in a polymer chain \cite{doi-edwards,1996-ottinger-book}, or as the ends of a full polymer chain \cite{bird,2005-celani,2005-afonso}. Having made this caveat, the canonical Langevin model for two spherical particles in a passive polymer system is given by
\begin{equation} \label{eq:general-model}
m \ddot \xv_i = - \nabla_{X_i} \Phi(\xv_1-\xv_2) + \zeta
(\uv(\xv_i(t),t) - \dot \xv_i(t)) + \kappa \dot \Wv(t)\;
\end{equation}
for $i = 1,2$. The mass $m$ is considered to be vanishingly small and so the
inertial term, $m \ddot \xv_i$, will be ignored. On the right hand
side, the first term is the restorative force exerted on the beads
due to the potential energy of the polymer's current configuration.
The function $\Phi$ denotes the configuration potential for the two beads.
The second term is an expression for the drag force exerted by a
time-dependent fluid velocity field $\uv$ with friction coefficient
$\zeta := 6 \pi a \eta$. This follows from the Stokes drag law for a
spherical particle of radius $a$ in a fluid with viscosity $\eta$.
The final term is the force due to thermal fluctuations in the
fluid where $W(t)$ is a standard Brownian motion. 
The diffusive constant $\kappa$ is often taken to be $\kappa
= \sqrt{2 k_B T \zeta}$, where $k_B$ is the Boltzmann constant and
$T$ is the temperature of the system in Kelvin, in accordance with
the fluctuation-dissipation theorem \cite{2005-celani}.

The goal of the present work is to achieve rigorous results about
the ergodicity of the \emph{connector} process
$$
\rv(t) := \frac{1}{2} (\xv_1(t) - \xv_2(t))
$$
in both $\kappa = 0$ and $\kappa \ne 0$ regimes with nonlinear
spring interaction in the presence of a spatially and temporally
correlated incompressible fluid velocity field.

\bigskip

In the simplest possible setting, one ignores the fluid and assumes a
Hookean (quadratic) spring potential $\Phi$. In this case, equation
\eqref{eq:general-model} is a simplification of the classical Rouse
model \cite{doi-edwards}. For the choice of 
$\Phi(r) = \frac{\gamma}{2}|r|^2$
the particle
dynamics satisfy the system of SDE
\begin{align*}
d \xv_1(t) &= \gamma \left[\xv_2(t) - \xv_1(t) \right] dt + \kappa \, d\Wv_1(t) \\
d \xv_2(t) &= \gamma \left[\xv_1(t) - \xv_2(t) \right] dt + \kappa
\, d\Wv_2(t)
\end{align*}
where $\Wv_1$ and $\Wv_2$ are independent standard Brownian motions.
The dynamics of the connector $\rv(t)$ are given by
$$
d \rv(t) = - 2 \gamma \rv(t) + \frac{\kappa}{\sqrt{2}} d \Wv(t).
$$
where $\Wv = \frac{1}{\sqrt{2}}(\Wv_1 - \Wv_2)$ is a standard Brownian
motion. We see that each of the connector components is an
Ornstein-Uhlenbeck process which therefore has the unique invariant
measure $ R^i(t) \sim N\left(0, \frac{\kappa^2}{8\gamma}\right).$
This exactly solvable model does not yield physical results, so one
must adopt nonlinear models for either or both of the spring
potential and fluid forces.

Significant theoretical advances exist for the dynamics of a single
tracer particle convected by a wide variety of fluid models
\cite{1999-kramer-majda}. One popular fluid model for
non-interacting two-point motions
\cite{2007-bec-heavy-particles-incompressible-flows}
\cite{2005-mehlig-aggregation} as well as for Hookean bead-spring
systems \cite{2000-chertkov-polymer-stretching,
2002-delucia-dummbbell-laminar-flow, 2005-celani} is a
time-dependent random field satisfying the statistics of the
Kraichnan-Batchelor ensemble \cite{1959-batchelor}
\cite{1968-kraichnan}. Such a fluid is still statistically white in time, but is
colored in space.

In the case where $\kappa = 0$ with non-interacting beads, the
spatial correlations in the convecting fluid velocity field allow
for concentration and aggregation phenomena
\cite{2002-sigur-preferential} \cite{2005-mehlig-aggregation}
\cite{2007-bec-heavy-particles-incompressible-flows}. This happens
because when the two beads are very close together, the fluid forces
on the respective beads are so strongly correlated there is no force
encouraging separation.

The presence of a diffusive term with $\kappa \ne 0$ prevents such
aggregation and the long term behavior of the connector depends on
so-called Weissenberg number $\text{Wi} = \zeta / 2 \gamma = \kappa^2 / 4 k_B T \gamma$
\cite{2005-celani}.  It is shown that when $\text{Wi} < 1$ the connector
$\rv$ will have a non-trivial stationary distribution, dubbed the
``coiled'' state. For $\text{Wi} > 1$, the connector does not have a
stationary distribution and is called ``stretched.''  The authors
express interest in the case where the fluid is not assumed to be
white-in-time.

\bigskip

In this work we use the incompressible stochastic Stokes equations to generate a fluid that is colored in space and time (see Section \ref{sec:theModel}). In the Hookean spring case (among other potentials with no repulsive force between the beads) with $\kappa = 0$, this model leads to degenerate dynamics (Proposition \ref{thm:hookean}). However, in a more general setting with a nonlinear spring potential that includes a repulsive force, we show that dynamics are nondegenerate, although the coiled / stretched state dichotomy discussed in \cite{2005-celani} is not present. We find that $\rv(t)$ is ergodic regardless of the physical parameters (Theorem \ref{thm:x-ergodicity}).

The method used here to establish ergodicity builds on the Harris
Chain theory developed in \cite{HarrisESMMP56,Hasminskii80,Num84}.
It is particularly indebted to the uniform ergodic results in
weighted norms developed in \cite{MeynTweedie93,MeynTweedie95}. The
argument follows the path outlined in \cite{2002-jcm-stuart, 2002-jcm-stuart-higham} for unique ergodicity of degenerate diffusions, but requires some nontrivial extensions to deal with the
multiplicative nature of the noise and to permit the type of
singular vector fields that arise as natural choices for the spring
potential $\Phi$.  We build a framework around a general ergodic
result from \cite{HairerMattingly08YAH} and then develop the needed
analysis to apply this framework.

Mathematically, as in \cite{2002-jcm-stuart-higham,2002-jcm-stuart},
this paper combines control theory with techniques from the theory
of hypoelliptic diffusions to invoke results in the spirit of
\cite{MeynTweedie93,MeynTweedie95}.  Ergodicity is obtained by
proving a minorization condition on a class of ``small sets'' (see
\cite{MeynTweedie93,MeynTweedie95}) while simultaneously 
establishing a matching
Lyapunov function. However, our problem has  a number of
difficulties which prevent the application of the results
\cite{2002-jcm-stuart-higham} directly. A central issue that needs to be
addressed is that the spring
potential, and hence the drift term, is permitted to have a singularity (Assumption
\ref{a:potential-singularity}). Therefore the natural candidates for
``small sets'' are not compact. This difficulty is overcome by
splitting the small sets into ``good'' and ``bad'' sets. On the
compact ``good'' set, defined in Eq.~\eqref{eq:defn-good}, we
demonstrate uniform  controllability as in
\cite{2002-jcm-stuart-higham,2002-jcm-stuart}. On the bad set, one
cannot obtain uniform control; however, the deterministic dynamics
move the system into the good set in finite time so that geometric
ergodicity still holds (Section~\ref{sec:ControlGoodBad}). Allowing
the spring potential to be singular extends the applicability of the
theory to many interesting, physically important potentials such as
the Lennard-Jones potential. Related ideas have been also recently
been used to prove ergodic and homogenization results in different
settings (see \cite{bubak09,HairerPardoux07}).

\subsection{Structure of paper and overview of results}

We will conclude Section 1 by proposing the model, leaving the proof of global existence and uniqueness to the Appendix.  It is important to point out that without a repulsive force between the beads, this model is degenerate. As an example, we consider in
Proposition~\ref{thm:hookean} a pair of particular choices -- including the Hookean spring model -- for the spring potential that do not introduce a repulsive force between the beads.  We find that the
distance between the beads $\rv(t)$ almost surely tends to $\zerov$
as $t \rightarrow \infty$ if the spring constant $\gamma$ is
sufficiently strong relative to a quantity that depends on the
typical spatial gradients in the random forcing.

In Section~\ref{sec:ergodicity}, we quote an abstract result from the classical ergodic theory literature. The quoted result requires proving a minorization condition and the existence of a Lyapunov function. Section~\ref{sec:measureTheoreticIrred} contains a general prescription for how to deduce the minorization condition from the existence of a continuous transition density and a weak form of topological irreducibility for the Markov process. In Section~\ref{sec:ControlGoodBad} the needed topological irreducibility is proven via a control theoretic argument. In Section~\ref{sec:Hormander} we invoke H\"ormander's ``sum of squares'' theorem to prove that the associated hypoelliptic diffusion has a smooth transition density. Section~\ref{sec:Lyap} contains the calculations establishing the existence of a Lyapunov function and Section~\ref{sec:generalizations} contains a number of generalizations and implications of the preceding results. The appendix contains the derivation of the model used.

\bigskip

Before preceding, we note that among the class of models we propose, the closest to that of Celani, et.~al. \cite{2005-celani} is the canonical Langevin Equation \eqref{eq:general-model} where the spring potential is quadratic, the mass $m$ is still 0, but the coefficient of the Brownian motion is nonzero: $\kappa = \sqrt{2 k_B T \zeta}$.  Our generalization is the replacement of the Kraichnan-ensemble with a finite-dimensional version of the stochastic Stokes equations. In this $\kappa > 0$ setting, the dynamics when $|\rv|$ is small become greatly simplified. Indeed, when the force separating the beads due to the fluid velocity becomes negligible, the remaining terms constitute an Ornstein-Uhlenbeck process. By standard ergodic properties of such processes, $\rv$ quickly leaves any small neighborhood of the origin with probability 1. For large values of $|\rv|$, the quadratic spring potential dominates and the Lyapunov function calculation we present in Section \ref{sec:Lyap} still holds. Since the diffusion is elliptic, existence of a continuous transition density follows trivially, and all arguments in the derivation of the stochastic $\delta$-ball controllability still apply, and thus the ergodic theorem we present in this work holds for $\rv(t)$.

This stands in contrast to the results in \cite{2005-celani} where it was argued that there exists a range of parameters where no stationary distribution exists. Furthermore, in light of the results we present here, it is not clear to us how to construct a model with colored-in-space-and-time fluid velocity field that supports the ``stretched'' and ``coiled'' regimes cited in the physics literature.  Unfortunately, we cannot comment directly on the model presented in \cite{2005-celani}, as our approach is highly dependent on the ability to express the dynamics in terms of a system of SDEs.

\subsection{Definition of the model}
\label{sec:theModel}

In the overdamped, highly viscous regime, it is reasonable to
neglect the nonlinear term in Navier-Stokes equations. Following
\cite{1989-ottinger-rabin}, \cite{2002-jcm-stuart},
\cite{2002-jcm-stuart-higham} and \cite{2002-sigur-inertial}, we
consider the bead-spring system advected by a random field $u: \rbb^2 \times \rbb \to \rbb^2$
satisfying the incompressible time-dependent stochastic Stokes
equations.  Following \cite{1986-walsh}, \cite{1992-daprato-zab},
\cite{1999-dalang} and \cite{2006-mckinley} the stochastic PDE
\begin{align}
\label{eq:defn-stokes} \partial_t \fluidv(x,t) - \viscosity\Delta
\fluidv(x,t) + \nabla p(x,t) &= F(d x, dt), \quad
 \nabla \cdot \fluidv(x,t) = 0
\end{align}
is well defined under the following conditions.  For technical
simplicity in the ergodicity arguments to come, we take $u$ to be spatially periodic with period $L$ which is presumed to be very large. We take the
space-time forcing $F:\rbb^2 \times \rbb \to \rbb^2$ to be a white-in-time, spatially periodic and colored-in-space Gaussian process satisfying
\begin{equation} \label{eq:defn-noise-stat}
\E{F(x,t)} = \zerov, \quad  \E{F^i(x,t) F^j(y,s)} = (t
\wedge s) \, 2 k_B T \nu \delta_{ij} \Gamma(x - y)
\end{equation}
where $\Gamma$ is the spatial covariance function, $\nu$ is the viscosity of the fluid, $t \wedge s$ denotes the minimum of $t$ and $s$, the component indices $i$ and $j$ are $i,j \in \{1, 2\}$ and $\delta_{ij}$ is a Kronecker delta function. 
As is shown in Appendix \ref{sec:model},
we may take the definition of the noise to be
\begin{equation} \label{eq:defn-noise}
F(x,t) = \frac{\sqrt{2 k_B T \nu}}{L} \! \sum_{\kv \in \zbb^2
\setminus \zerov} \! \left(\cos\big(\lambda \kv \cdot x\big) B_{k}^1(t) +  \sin\big(\lambda \kv \cdot x\big) B_{k}^2(t)\right) \! \sigma_{\kv} \;
\end{equation}
where we have introduced the inverse length scale $\lambda = 2 \pi / L$ and the $B_{\kv}^i$ are independent standard 2-$d$ Brownian
motions. The coefficients $\sigma_\kv$ are related to the spatial
correlation function $\Gamma$ through the Fourier relation $
\Gamma(x) = \frac{2}{L^2} \! \sum_{k \in \zbb^2 \setminus
\{\zerov\}} \! \! \cos\big(\lambda k \cdot x\big) \sigma_{\kv}^2.$

This relation is possible because $\Gamma$ is a covariance function,
and therefore positive definite. By Bochner's Theorem, $\Gamma$ is
realizable as the Fourier inverse transform of a positive real valued measure
called the \emph{spectral} measure. Often one defines the correlation structure on the spectral domain. 
For clarity of exposition, we take the set of modes with nonzero $\sigma_{\kv}$, denoted $\modeset \subset \zbb^2\setminus (0,0)$ to be finite but containing at least three linearly independent vectors. We use $N = |\modeset|$ to denote the number of active modes. 

As is discussed in the Appendix, Section \ref{sec:model}, we can express the dynamics of the eigenmodes in terms of the family of independent 1-dimensional Ornstein-Uhlenbeck processes $Z(t) := \{Z_k(t)\}_{k \in \modeset}$ respectively satisfying
\begin{equation} \label{eq:defn-z}
d Z_\kv(t) = - \length^2 \viscosity \kvsq Z_\kv(t) dt + \sqrt{2
\beta \nu} \length \sigma_\kv \,d W_\kv(t)
\end{equation}
where $\beta = k_B T / 4 \pi^2$ and $\{W_k\}_{\kv \in \modeset}$ is a family of iid standard 1-dimensional Brownian motions. For each $k$, we take the
initial condition $Z_k(0)$ to be chosen from its respective stationary
distribution, namely $Z_k(0) \sim \text{N}\left(0, \beta \sigma_\kv^2 / \kvsq \right).$

Our goal will be to rigorously analyze the long-term behavior
of the connector process $\rv$ whose dynamics we will study via an approximate system which is derived in the Appendix, Section \ref{sec:model}.  This entails writing $X_1$ and $X_2$ in terms of the configuration vector $R(t)$  and the ``center of mass'' process $M(t)$.  As is discussed in that development, we set $M(t)=0$ to substantially simplify subsequent calculations.  We argue that this assumption can be removed and that all of the relevant results hold for the original system.  

We now define our model for $R(t)$.  Given the family $Z(t)$ defined by \eqref{eq:defn-z}, let $R:\rbb \to \rbb^2$ satisfy the time-inhomogeneous ODE
\begin{equation} \label{eq:defn-r}
\ddt \rv(t)  = - \nabla \Phi(\rv(t)) + \sum_{\kv \in \modeset} \sin\!\left(\length \, \kv \cdot R(t)\right) \frac{\kvperp}{|\kv|} Z_k(t)
\end{equation}
where for a given vector $\kv = (k_1, k_2)$ we denote $\kvperp := (-k_2,k_1)$. The configuration potential $\Phi: \rbb^2 \to \rbb$ is discussed below in Assumption~\ref{a:potential-singularity}. The last term of \eqref{eq:defn-r} summarizes the influence of the fluid on the separation between the beads.  We will write this in terms of the multiplication of the $2\times N$ Stokes matrix $S(r)$ by the vector $z = (z_1, \ldots z_N)$, 
\begin{equation} \label{eq:defn-stokes-matrix}
	S(r)z := \sum_{\kv \in \modeset} \sin\!\left(\length \, \kv \cdot r\right) \frac{\kvperp}{|\kv|} z_k\,.
\end{equation}
We discuss the existence and uniqueness of the ODE \eqref{eq:defn-r} in Appendix \ref{sec:exist-unique} and will think of the solution $R$ with initial condition $r_0$ in terms of the mapping  
\begin{equation} \label{defn:Psi}
	R := \Psi(r_0,Z)
\end{equation}
where $\Psi: \rbb^2 \times C([0,\infty),\rbb^N) \to C([0,\infty),\rbb^2)$ is the solution of the ODE given in \eqref{eq:defn-r}.

As mentioned earlier, the choice of quadratic potential $\Phi$
corresponds to a Hookean spring model. There are a number of canonical choices for nonlinear spring potentials (see \cite{bird} Table 10.1-1) but of particular interest to us are potentials which only allow for a finite maximum extension of the polymer.  One common choice is known as the \emph{finite extensible nonlinear elastic} (FENE) \cite{bird,2005-afonso,2003-thiffeault} potential:
\begin{equation} \label{eq:defn-fene}
	\Phifene(r) = \frac{\gamma \rhomax^2}{2} \ln\left(\frac{1}{1 - |r|^2/\rhomax^2}\right).
\end{equation}
The parameter $\rhomax > 0$ is the maximal extension of the chain.
However, because there is no repulsive force in the potential, we find that systems with these potentials have degenerate dynamics (Proposition \ref{thm:hookean}). 
In the sequel, we place the following assumptions on the spring potential.

\begin{assumption} \label{a:potential-singularity}
Let $0 < \rho_{max} \leq \infty$ be given and define
$$
\dcal := \{r \in \rbb^2 \mbox{ such that } |r| \leq \rho_{max}\}.
$$
We assume that the spring potential $\Phi: \dcal \to \rbb_+$ satisfies $\Phi(0) = 0$ and each of the following conditions.
\begin{enumerate}
	\item \emph{Radial symmetry.}  For some continuously differentiable function $\phi: (0,\rhomax) \to \rbb_+$, we have
	\begin{equation} \label{eq:defn-phi-radial}
		\Phi(r) = \phi(|r|).
	\end{equation}
	\item \emph{Locally Lipschitz gradient.} For any compact region $K \subset \dcal \setminus \{0\}$, there exists a constant $C > 0$ such that for all $r_1, r_2 \in K$,
	$$
	|\nabla \Phi(r_1) - \nabla \Phi(r_2)| \leq C |r_1 - r_2|.
	$$
	\item \emph{Compact level sets.} For every $\rho \geq 0$, the set $ \{r \in \dcal \mbox{ s.t. } \Phi(r) \leq \rho\} $ is compact.
	\item \label{cond:K} \emph{Growth condition.} The potential satisfies $\displaystyle \lim_{|r| \to \rhomax} \Phi(r) = \infty $
	and there exists a $\gamma > 0$ and a $\rho_0 < \rhomax$ such that for all $r \in \dcal$ with $|r| \in (\rho_0,\rhomax)$ 
	\begin{equation} \label{eq:spring-large-r}
		|\nabla \Phi(r)|^2 \geq \gamma \Phi(r).
	\end{equation}
	\item \label{cond:repulsive}\emph{Repulsive force at the origin.} There exists $\gamma_0 > 0$ and $\epsilon_0> 0$ such that for all
	$r \in \dcal \setminus \{0\}$ with $|r| \leq  \epsilon_0$
	\begin{equation} \label{eq:spring-small-r} 
		- \nabla \Phi(r) \cdot r \geq \gamma_0.
	\end{equation}
\end{enumerate}
\end{assumption}
\begin{remark}
It is in this context that we choose the length of the periodicity of the forcing fluid. We take $L \gg 4 \rho_0$.
\end{remark}

We have in mind potentials that consist of standard choices when the beads are separated by large distances, but that have a singularity at zero. For example, the above assumptions include the
families of functions
\begin{equation} \label{eq:defn-lennard-jones}
\Phi(r) = \frac{1}{2q} |r|^{2q} + \frac{1}{\alpha |r|^\alpha}, \quad \text{ and } \qquad 
	\Phi(r) = \Phifene(r) + \frac{1}{\alpha |r|^\alpha}.
\end{equation}
where $\alpha$ is a positive constant. The choice $\alpha = 12$ corresponds to a Lennard-Jones singularity at zero.  One can check that the Growth Condition \eqref{cond:K} is satisfied for such potentials if and only if $q \geq 1$.

\section{Ergodicity}\label{sec:ergodicity}

In order to state our main result, we must set some notation. Let $\Xv(t) =
(\rv(t),\zv(t))$ satisfy the system given by \eqref{eq:defn-z} and \eqref{eq:defn-r}.  It follows from Proposition \ref{prop:exist-unique} that the process $\Xv(t)$ is
Markov and well-defined on the state space
\begin{equation*}
  \xbb := \big\{(r, z) \in \dcal \times \rbb^N\big\}\;.
\end{equation*}
For a bounded, measurable function $\f: \xbb \to \rbb$, we define the
action of the Markov semigroup $\cP_t$ by
$$
(\cP_t \f)(x) = \Ex{x}{\f(\Xv(t))}\;.
$$

To measure convergence to equilibrium, we introduce the following
weighted norm on such functions $\f$ relative to a given Lyapunov
function $V:\xbb \to [0,\infty)$,
$$
\|\f\| := \sup_{x \in \xbb} \frac{|\f(x)|}{1 + V(x)}\,.
$$
We note that the Markov semigroup $\cP_t$ can be extended to act on all functions $\f$ bounded pointwise above by $V$. Henceforth, we will use  
\begin{equation} \label{eq:defn-lyapunov}
V(x) := \moll(\Phi(r)) + \eta |z|^2 \,
\end{equation}
as the Lyapunov function for the Markov process $X(t)$, where $\psi:\rbb \to \rbb$ is the function 
\begin{equation} \label{eq:defn-moll}
	\moll(x) := \left\{\begin{array}{cl} 0, & 0 \leq x \leq a \\ c \, (x - a) \,e^{-1/(x-a)^2}, & x > a
\end{array} \right.,
\end{equation} 
where we set $a = \phi(\rho_0)$. The constant $\rho_0$ is as in Equation \eqref{eq:spring-large-r} of Assumption \ref{a:potential-singularity}, and the constants $c$ and $\eta$ are set by an argument in Section \ref{sec:Lyap}. The essential properties of $\psi$ are recorded in Section \ref{sec:psi}.

The main result of this article is the following statement about the geometric ergodicity of the Markov process $\Xv$, which in turn implies the connector process $\rv$ converges to its unique non-trivial stationary distribution in exponential time.

\begin{theorem} \label{thm:x-ergodicity}
Suppose that the set of active modes $\modeset$ is finite, but
contains at least three pairwise linearly independent vectors, and
let the spring potential $\Phi$ satisfy Assumption
\ref{a:potential-singularity}. Then there exists a unique
non-trivial invariant measure $\pi$  and constants $C >0$ and
$\lambda > 0$ so that for all measurable $\f:\xbb\to \rbb$ with
$\|\f\| < \infty$, we have
\begin{align*}
    \|\cP_t \f - \pi\f \| \leq C e^{-\lambda t} \|\f\|
\end{align*}
where $\pi \f = \int \f d\pi$.
\end{theorem}

Let us introduce a family of weighted $L^\infty$-norms that
depend on a scale parameter $\beta > 0$. For a measurable $\f\colon\xbb
\rightarrow \rbb$ define
\begin{align*}
  \norm{\f}_{\beta} := \sup_{x \in \xbb} \frac{|\f(x)|}{1 + \beta V(x)}\,.
\end{align*}
Observe that $\|\cdot\|_1 = \|\cdot\|$ and any two norms in this family are equivalent.
Define the corresponding dual metric on probability measures:
\begin{align*}
  \dd_{\beta}(\mu_1,\mu_2) &= \sup_{\f : \norm{\f}_\beta \leq 1} \int
  \f(x) \mu_1(dx) - \int
  \f(x) \mu_2(dx)
\end{align*}
for two probability measure $\mu_1, \mu_2$ probability measures on $\xbb$. Note that $\dd_\beta$
is the usual total variation norm for $\beta = 0$. 
Theorem \ref{thm:x-ergodicity} follows from classical results in \cite{MeynTweedie93} and \cite{MeynTweedie95} adapted to our
setting:

\begin{theorem}\label{thm:HM-harris} Suppose that  the Lyapunov function $V \colon \xbb \to [0,\infty)$
has compact level sets with $\lim_{x \to \partial \mathbb{X}} V(x) =
\infty$ and that  for some $t > 0$, $c_1 >0$
and $c_0 \in (0,1)$, it satisfies
\begin{equation} \label{eq:lyapunov-condition}
(\cP_t V)(x) \leq c_0 V(x) + c_1
\end{equation}
for all $x \in \xbb$. (Here, the boundary set $\partial \mathbb{X}$
includes the point at infinity in unbounded directions.) Furthermore suppose there exists a
probability measure $\nu$ and constant $\alpha \in (0,1)$ such that
\begin{equation} \label{eq:minorization-condition}
  \inf_{x \in \cpt} \cP_t(x, \ccdot) \geq \alpha \nu(\ccdot)
\end{equation}
with $\cpt := \{x \in \xbb \colon V(x) \leq  K\}$ for some $K \geq
2c_1/(1 - c_0)$.

Then there exists an $\alpha_0 \in (0,1)$ and $\beta > 0$ so that
  \begin{align*}
    \dd_\beta( \cP_t^* \mu_1 , \cP_t^* \mu_2) \leq \alpha_0\,
    \dd_\beta(\mu_1,\mu_2)
  \end{align*}
  for any two probability  measures $\mu_1$ and $\mu_2$ on $\xbb$.
\end{theorem}

We begin by fixing the set $\cntr$ which should be thought of as the ``center'' of the
state space.  At the end of the proof of Lemma \ref{lem:lyapunov} we
select a value $\rhogood \in (\rho_0,\rhomax)$ which is used to define
$\cntr$: 
\begin{align} \label{eq:defn-center}
  \cntr := \{x \in \xbb : V(x) \leq \psi(\phi(\rho_+))\}.
\end{align}
Recall that the Lyapunov function $V$ is defined by \eqref{eq:defn-lyapunov} with $\phi$ and $\moll$ defined by \eqref{eq:defn-phi-radial} and \eqref{eq:defn-moll}, respectively.
As is established by the following lemma, $V$ satisfies the
inequality \eqref{eq:lyapunov-condition}. We defer the somewhat
standard proof of this lemma to the appendix, Section \ref{sec:Lyap}.
\begin{lemma}[Lyapunov function] \label{lem:lyapunov}
  Fix the values of the constants $\eta$, $c$ and $\rhomax$ so that they satisfy the constraints imposed by the inequalities \eqref{eq:defn-eta}, and let $V(x)$ be defined as in \eqref{eq:defn-lyapunov}.
Then for any $t \geq 1$ there exist constants $c_0 := c_0(t) \in (0,1)$ and $c_1  := c_1(t) \geq 0$ such that \eqref{eq:lyapunov-condition} holds. Moreover we have $\psi(\phi(\rhogood)) \geq \frac{2c_1}{1- c_0}$ as required for the definition \eqref{eq:defn-center} of $\cntr$ by Theorem \ref{thm:HM-harris}.
\end{lemma}
The remainder of Section \ref{sec:ergodicity} is concerned with constructing a minorizing measure, as required by condition \eqref{eq:minorization-condition}.  The main result is Proposition \ref{thm:two-conditions}. Its proof follows from the topological irreducibility of the transition semigroup established in Proposition \ref{lem:top-irred-C} and the ``local smoothing'' property proved in Proposition \ref{lem:meas-irr}.  The local smoothing property follows from hypoellipticity of the generator of the Markov process $\Xv$ and a version of H\"ormander's sum of squares theorem (cf. \cite{hormander85III,stroock08PDE}).

\subsection{Conditions for measure-theoretic irreducibility}\label{sec:measureTheoreticIrred}

In this section we  use a very weak form of topological irreducibility to prove the measure-theoretic minorization and irreducibility required in \eqref{eq:minorization-condition}.

\begin{prop} \label{thm:two-conditions}
	Suppose there exists an $x_* \in \cntr$ such that the following two conditions hold.  Then there exists a constant $\alpha \in (0,1)$, a time $t \geq 1$ and a probability measure $\nu$
such that \eqref{eq:minorization-condition} holds.
  \begin{enumerate}
  \item \label{cond:stoch-control} \emph{Uniformly Accessible Neighborhood Condition:} For any $\delta>0$ there exists a constant $r > 0$ and a positive function $\alpha_0 \colon (0,\infty) \rightarrow (0,\infty)$ such that
    \begin{align}\label{eq:topIrrMinorize}
      \inf_{x \in \cntr}\cP_{r'}(x,B_\delta(x_*)) \geq \alpha_0(r')
    \end{align}
    for all $r' > r$.
  \item \emph{Continuous Density Condition:} There exists an $s >0$ and an open set $\ocal \subset \cntr$ with $x_* \in \ocal$, such that for any $x \in \ocal$ and measurable $A \subset \ocal$ one has
    \begin{align*}
      \cP_{s}(x,A)=\int_A p_{s}(x,y)dy
    \end{align*}
    with $p_{s}(x,y)$ jointly continuous in $(x,y)$ for $x,y \in \ocal$ and
    $p_s(x_*,y_*)>0$ for some $y_* \in \ocal$.
  \end{enumerate}
\end{prop}
\begin{proof}
  By the continuity assumption on $p_s$ there exists $\delta > 0$ so that
  $B_\delta(x_*),B_\delta(y_*) \subset \ocal$ and
  \begin{align*}
    \inf_{x \in B_\delta(x_*)}\inf_{y \in
      B_\delta(y_*)} p_{s}(x,y) \geq \frac12 p_{s}(x_*,y_*) >0\;.
  \end{align*}
  We define the minorizing probability measure $\nu$
  by $\nu(A) = {\lambda(A\cap B_\delta(y_*))}/{
  \lambda(B_\delta(y_*))}$
  where $\lambda$ is Lebesgue measure and $A$ is any measurable set.
With this $\delta$ we also fix $r = r(\delta)$ according to the Uniformly Accessible Neighborhood Condition \eqref{cond:stoch-control}.

Now, pick $t \geq 1 + r + s$ and define $\alpha(t)=\frac12 (1 \wedge p_{s}(x_*,y_*) \alpha_0(t-s)\lambda(B_\delta(y_*)))$, where $\alpha_0$ is the function given in \eqref{eq:topIrrMinorize}. Then for any measurable set $A$ and $x_0 \in \mathcal{C}$ we have
  \begin{align*}
    \cP_t(x_0,A) &=  \int_A \int_{\rbb^{2+N}} \cP_{t-s} (x_0, dx) \cP_{s} (x,dy)\\
    &\geq \int_{A \cap B_\delta(y_*)} \left(\int_{B_\delta(x_*)} \cP_{t-s}(x_0,dx)\right) p_{s}(x, y) dy \\
	&\geq \int_{A\cap B_{\delta}(y_*)} \alpha_0(t-s)\frac{1}{2} p_s(x_*,y_*) dy \geq \alpha(t) \nu(A)\;,
  \end{align*}
which proves the claim.

\end{proof}

\subsection{Topological irreducibility}
\label{sec:ControlGoodBad}

This section is devoted to proving the Uniformly Accessible Neighborhood Condition \eqref{cond:stoch-control} stated in Proposition \ref{thm:two-conditions}. This argument consists of first proving that under the spring potential conditions listed in Assumption \ref{a:potential-singularity}, the system has non-trivial long-term behavior.  Unlike the Hookean spring case where the two particles come together as $t \to \infty$ almost surely (Proposition \ref{thm:hookean}), in the non-linear (with repulsion) spring case we can show that two particles arbitrarily close together have a positive probability of separating in an explicitly defined finite time (Lemma \ref{lem:nonlinear-near-zero}). Given this separation property, we employ a control argument to show the noise has a positive probability of directing the system to a neighborhood of a specified reference point $x_* \in \cntr$ (Lemmas \ref{lem:top-irred} and \ref{lem:top-irred-C}). 

\subsubsection{A particle separation lemma}

\begin{lemma} \label{lem:nonlinear-near-zero}
Let $M > m > 0$ be given, suppose $(\rv(0),\zv(0)) =
(r_0,z_0) \in \dcal \times \rbb^N$, and define
$$
\tau_\epsilon(r_0, z_0) := \inf\{t \geq 0: |\rv(t)| \geq \ep
\mbox{ and } |\zv(t)| < M\}.
$$
Then there exists an $\ep \in (0,\ep_0]$ where $\ep_0$ is defined in Assumption \ref{a:potential-singularity} and an $\alpha \in (0,1)$ such that $\tau_\ep$
satisfies
\begin{equation} \label{eq:escape-lower-bound}
\inf_{\{z_0 : |z_0| < m\}} \inf_{\{r_0 : 0 < |r_0| \leq
\ep\}} \p{\tau_\ep(r_0, z_0) \leq 1} \geq \alpha.
\end{equation}
\end{lemma}

\begin{proof}
The essence of the argument is that if the noise stays relatively small for
sufficiently long, then the repulsive force will dominate the
$\rv$-dynamics and force the particles away from each other. Without loss of generality, for the remainder of this proof we assume that the initial condition $(r_0,z_0)$ satisfies $r_0 \leq \ep_0$ and $|z_0| \leq m$. 

We denote the event that the magnitude of the noise stays moderate by
$
\Omega_{z} := \big\{\sup_{t \in [0, 1]} |\zv(t)| < M\big\}$ and claim there exists an $\ep \in (0,\ep_0]$ and $\alpha > 0$ such that
$\p{\Omega_{z}} \geq \alpha$ and $\p{\tau_\ep \leq 1 \, | \,
\Omega_{z}} =1$ and therefore
\begin{equation*} 
\p{\tau_\ep \leq 1} \geq \p{\tau_\ep \leq 1 \, | \,
\Omega_{z}}\cdot \p{\Omega_{z}} \geq 1 \cdot \alpha.
\end{equation*}
We first prove that there exists an $\alpha > 0$ such that 
\begin{equation} \label{eq:separation-claim-one}
	\inf_{z_0 : |z_0| \leq m} \p{\Omega_{z}} \geq \alpha.
\end{equation}
Indeed, the noise vector $Z(t) = (Z_1(t), Z_2(t), \ldots, Z_N(t))$ can be written
\begin{equation} \label{eq:z-duhamel}
\zv(t) = e^{-\Lambdav t} z_0 + \int_0^t e^{-\Lambdav(t
- s)} B d\Wv(s)
\end{equation}
where $\Lambdav$ is a diagonal matrix whose entries $\{\lambda_\kv\}_{\kv \in \modeset}$ are given by $\lambda_\kv := \lambda^2 \nu |\kv|^2$ and $B$ is a diagonal matrix whose entries $\{b_k\}_{k\in \modeset}$ are given by $\sqrt{2 \beta \nu} \length \sigma_\kv$.

It follows from \eqref{eq:z-duhamel} that
$$
|\zv(t)| \leq m + \sum_{\kv \in \modeset} \left|e^{-\lambda_\kv t}
\int_{0}^t e^{\lambda_{\kv} s} b_k d W_{\kv}(s)\right|.
$$
Since $M_\kv(t) := \int_{0}^t e^{\lambda_{\kv} s} b_k d W_{\kv}(s)$ is a continuous martingale with quadratic variation $\langle M_\kv, M_\kv \rangle_t = b_k^2 (e^{2 \lambda_\kv t} - 1)/ 2 \lambda_{\kv}$, then for any $t > 0$, $M_\kv(t)$ has the same distribution as $\tilde W( \langle M_\kv, M_\kv \rangle_t )$ where $\tilde W$ is a standard Brownian motion. It follows that 
\begin{align*}
\alpha_{\kv} &:= \mathbb{P}\Big\{\sup_{t \in [0, 1]}
\big|e^{-\lambda_\kv t} \int_{0}^t e^{\lambda_{\kv} s} d
W_{\kv}(s)\big| \leq \frac{M - m}{N}\Big\} \geq  \mathbb{P}\Big\{\sup_{t \in [0,t_k]} |\tilde W(t)| \leq \frac{M - m}{N}\Big\}
\end{align*}
where $t_k = b_k^2 (e^{2 \lambda_\kv } - 1)/
2 \lambda_{\kv}.$
Since a Brownian motion will stay within a prescribed tube over an
arbitrarily long finite interval with positive probability, we have
that $\alpha_\kv > 0$.
Because there are only finitely many modes and they are mutually
independent, we have $\p{\Omega_{z}} \geq \prod_{\kv \in
\modeset} \alpha_{\kv} > 0.$ To conclude the proof of the claim \eqref{eq:separation-claim-one}, it remains only to note that this lower bound for $\p{\Omega_z}$ does not depend on the initial condition $z_0$ as long as $|z_0|\leq m$.

We now show that there exists an $\ep > 0$ so that
\begin{equation} \label{eq:escape-conditioned-omega-0}
\p{\tau_\epsilon \in [0,1] | \, \Omega_{z}} = 1.
\end{equation}
Let $\ep_0$ and $\gamma_0$ be the positive constants from
\eqref{eq:spring-small-r} of Assumption
\ref{a:potential-singularity}. We fix
\begin{equation} \label{eq:defn-escape-ep}
\ep := \ep_0 \wedge \sqrt{\frac{(1 - e^{-(N M)^2 / \gamma_0})}{(N M)^2}}
\end{equation}
and define $\sigma_\epsilon := \inf\{t \geq 0: |\rv(t)| \geq \ep\}$. Conditioned on the event $\Omega_z$,
we have $\tau_\epsilon=\sigma_\epsilon$, and so to prove \eqref{eq:escape-conditioned-omega-0} it suffices to show $\sigma_\epsilon \leq 1$ on $\Omega_{z}$.

Recall the ODE \eqref{eq:defn-r} defining $R$ and the notation $S$ for the Stokes matrix, see \eqref{eq:defn-stokes-matrix}. For any $t \in [0,\sigma_\ep]$ and for any $\vartheta >
0$ we have the differential inequality
\begin{align*}
\ddt \frac{1}{2} |\rv|^2 &= - \nabla \Phi (\rv) \cdot \rv + (S(\rv)
\zv)
\cdot \rv 
\geq \gamma_0 - \vartheta |S(\rv) \zv|^2 - \frac{1}{4\vartheta} |\rv|^2
\end{align*}
where we have applied the inequality \eqref{eq:spring-small-r} from
Assumption \ref{a:potential-singularity} to the first term and the
polarization inequality $x \cdot y \geq -(\vartheta |x|^2 +
\frac{1}{4\vartheta} |y|^2)$ to the second term. Furthermore $
|S(\rv) \zv| \leq \|S(\rv)\|_F |\zv|$ where $\|\cdot\|_F$ is the
matrix Frobenius norm.  The contribution of each column (respectively associated to an eigenmode $k$) of the
Stokes matrix to its Frobenius norm is exactly $\sin^2(\lambda \kv
\cdot \rv)$. It follows that $\|S(\rv)\|_F \leq N
$. Hence for all $t \in [0, \sigma_\ep]$,
\begin{align*}
\ddt \frac{1}{2} |\rv(t)|^2 &\geq -\frac{1}{4\vartheta} |\rv(t)|^2 +
(\gamma_0 - \vartheta N^2 |\zv(t)|^2).
\end{align*}
Restricting to the event $\Omega_{z}$ and fixing $\vartheta =
\gamma_0 / 2 (N M)^2$, we have
\begin{equation*}
  \ddt |\rv(t)|^2 \geq - \frac{(N M)^2}{\gamma_0} |\rv(t)|^2 + \gamma_0.
\end{equation*}
For any $t \in [0,1]$,  integrating the preceding estimate on
$\Omega_{z}$ yields
\begin{align*}
  |\rv(t \wedge \sigma_\epsilon)|^2 &\geq e^{-(t \wedge \sigma_\epsilon) (N M)^2 / \gamma_0 } |r_0|^2 + \gamma_0 \int_{0}^{t \wedge \sigma_\epsilon} e^{-[(t \wedge \sigma_\epsilon)-s] (N M)^2 /\gamma_0} ds \\
&\geq \frac{\gamma_0^2}{(N M)^2} \left(1 - e^{-(t \wedge \sigma_\epsilon)
(N M)^2 / \gamma_0} \right)\,.
\end{align*}
We want to show that on $\Omega_{z}$, $\sigma_\epsilon \leq 1$
with probability one. Suppose that $\sigma_\epsilon >  1$. Then the last
estimate implies that
$$
|\rv(1 \wedge \sigma_\epsilon)|^2 = |\rv(1)|^2 \geq (N
M)^{-2}(1 - e^{-(N M)^2 /\gamma_0}) \geq \ep^2
$$
and hence $\sigma_\epsilon\leq 1$.
We conclude the claim \eqref{eq:escape-conditioned-omega-0}, which completes the proof.
\end{proof}

\subsubsection{Topological irreducibility via control}

By Assumption \ref{a:potential-singularity}, the spring potential $\Phi$ has a
(possibly non-unique) global minimum $r_{\min}$, which satisfies
$|r_{\min}| \leq \rho_0$ where $\rho_0$ was the constant from Assumption~\ref{a:potential-singularity}. 
We choose a global minimum closest to the origin and denote it by $r_{\ast}$.
Since the global
minimum of the noise norm $|\cdot|$ is achieved at the origin, $z_*
= 0$, we set the global reference point
\begin{equation} \label{eqn:pointxstar}
x_\ast := (r_\ast,0)
\end{equation}
which is a minimum of the Lyapunov function $V$.

We wish to use the $\zv$ process to drive the $\rv$ process to the reference point $r_\ast$.
However, due to the possible singularity at the origin (see Assumption
\ref{a:potential-singularity}) the differential equation
\eqref{eq:defn-r} for $\rv$ may have unbounded coefficients which presents a genuine
difficulty in applying control theoretic arguments. We
therefore will designate a region of bad control, $\mathcal{B}$, within the center
$\cntr$ (see \eqref{eq:defn-center}), as well as a compact region of good control, $\good$.

In Lemma \ref{lem:nonlinear-near-zero} we demonstrated that the $R$ process has a positive probability of escaping from a neighborhood of 0 in unit time. 
Let $\ep_1$ be the constant derived from applying Lemma \ref{lem:nonlinear-near-zero}
with $m = \psi(\phi(\rhogood))$ and $M = m/\sqrt{\eta}$,
where $\eta$ is given in \eqref{eq:defn-eta}. Since $\eta \leq 1/2$ we have 
 $M > m > 0$ as required by the hypothesis of Lemma \ref{lem:nonlinear-near-zero}. We define the set of ``bad'' points in $\cntr$ by
\begin{align}
  \label{eq:defn-bad}
  \bad=\big\{ (r,z) \in \cntr : |r| < \ep_1\big\}\;.
\end{align}
Next, we define the set of ``good'' points $\good$ to be
\begin{equation}
\label{eq:defn-good} \good = \rgood \times \zgood := \Big\{(r,z)
\in \xbb : |r| \in \big[\ep_1, \rhogood \big] , |z|^2 \leq \psi(\phi(\rhogood))/\eta
\Big\}\;.
\end{equation}
Note that $\cntr \subset \good \cup \bad$.

We now  use a controllability argument to establish the weak form of
uniform topological irreducibility on $\good$ given (for the set
$\cntr$) in Eq. \eqref{eq:topIrrMinorize}. 
\begin{lemma}[Topological irreducibility on the ``good'' set $\good$]
\label{lem:top-irred} Let  $x_\ast \in
\cntr$ be as given in \eqref{eqn:pointxstar}. Then for  any
  $\delta > 0$ there exists  $t_1 >0$ so that for any $t_2 > t_1$  there exists $\alpha_1 > 0$ 
  such that
  \begin{align}\label{eq:goodMinor}
  \inf_{t \in [t_1,t_2]}  \inf_{x \in \good} \tkernel_{t}(x, B_{\delta}(x_*)) \geq \alpha_1 .
  \end{align}
\end{lemma}

The proof  of the above lemma relies on the following three observations, whose proofs are 
deferred to the appendix. In what follows, for $f: I \subset \mathbb{R} \mapsto \mathbb{R}^n$,
define the  sup-norm 
$$|f|_\infty := \sup_{t\in I} |f(t)|\, .$$ 

The first observation is that there is a bounded deterministic control $\tilde \zv$ that accomplishes the task of moving its associated connector $\tilde R = \Psi(r_0,\tilde Z)$ (recall the definition in Equation \eqref{defn:Psi}) from the initial position $r_0$ to the reference point $r_*$ at time $t = 1$.
\begin{fact} \label{control-list-one}
\emph{(Existence of a deterministic control.)} 
For any initial position $\tilde r_0 \in \rgood$, the set $\rcal \subset
C^\infty([0,1];\rgood)$ defined by
\begin{equation}\label{eq:R}
 \rcal := \Big\{\tilde R \ : \
 \tilde \rv(0) = \tilde r_0, \, \tilde \rv(1) = r_\ast,
 \Big|\frac{d\tilde \rv}{dt} \Big|_\infty \leq 5 \rhogood \Big\}
\end{equation}
is non-empty.  Furthermore, there exists an $M_1 > 0$, which does not depend on $\tilde{r}_0$, such that for
any $\tilde \rv~\in~\rcal$, there exists a continuous $\tilde
\zv \in C([0,1]; \rbb^N)$ such that
\begin{align*}
\tilde  \rv= \Psi(r_0,\tilde \zv) \quad \text{ and
}\quad |\tilde \zv|_\infty \leq M_1\,.
\end{align*}
\end{fact}

Next we notice that the map $(r, Z) \mapsto \Psi(r,Z)$ is continuous when $r$ belongs to the 
good set $\good$.  For $\tilde Z \in C([0,T]; \rbb^N)$ and constants $M, \gamma, \delta_z > 0 $, define
the set 
\begin{align} \label{eqn:fancyz2}
\mathcal{Z}(\tilde Z, M,\gamma, \delta_z):=  \Big\{  \zv \  :\   |\zv(t) -  \tilde \zv(t)|
  \leq M e^{-\gamma t} + \delta_z  \quad  \forall \ t \in [0,T] \ \Big\}.
\end{align}
\begin{fact} \label{control-list-two}
\emph{(Continuity of the map $\Psi$.)}
Fix any $\tilde r_0 \in \rgood$, $M_2, \, T > 0$ and  $\delta_r \in (0,\ep_1/2)$ where  $\ep_1$  is from \eqref{eq:defn-bad}. Suppose that $\tilde \zv \in C([0,T]; \rbb^N)$ satisfies $|\tilde \zv|_\infty \leq M_2$. 
Then there exist constants $\gamma > 0$, $\delta_0 > 0$ and $\delta_z >0$ such that
 $$|\Psi(r_0,\zv) -  \Psi( \tilde r_0, \tilde \zv)|_\infty \leq \delta_r$$ 
 for all $
(r_0, \zv)  \in \Big \{ \good_r \cap \{r: |r- \tilde r_0| \leq \delta_0  \} \Big \} \times  \mathcal{Z}(\tilde \zv, M_2, \gamma, \delta_z).$
\end{fact}
Finally, we observe that OU processes stay in a tubular neighborhood with positive probability.
\begin{fact}
 \label{control-list-three}
\emph{(Approximation by OU processes.)}
Let a set $\mathcal{Z} = \mathcal{Z}(\tilde Z, M, \gamma, \delta_z)$ be given. Then there exists a $p > 0$ such that
\begin{align*}
  \inf_{z_0 \in \good_z} \pbb_{z_0}\big\{ \zv \in \mathcal{Z} \big\}
\geq p\,
\end{align*}
where $Z = (Z_1, \ldots, Z_N)$ is the solution to \eqref{eq:defn-z} with $Z(0) = z_0$.
\end{fact}

With these observations we now prove Lemma \ref{lem:top-irred}.
\begin{proof}[Proof of Lemma \ref{lem:top-irred}] 
Fix an initial condition $x_0 = (r_0,z_0) \in \good$
and $\delta > 0$.  The argument proceeds in two steps. First we construct a
bounded deterministic control $\tilde \zv$ that accomplishes
the task of moving its associated connector $\tilde R = \Psi(r_0,\tilde Z)$ from the initial position $r_0$
to the reference point $r_*$ at time $t = 1$. Any instance of the
noise $\zv$ that approximates $\tilde \zv$
sufficiently well, as in the definition of $\mathcal{Z}$ above, will
have an associated connector $\rv = \Psi(r_0,\zv)$
that has a terminal position $\rv(1)$ near $r_\ast$. Demonstrating that such an event has positive probability is not sufficient to prove \eqref{eq:goodMinor}. This is because $\zv(1)$ may not be close to $Z_\ast = 0$. Therefore in the second step of the proof we show that, conditioned on success during the time interval $t \in [0,1]$, the noise has a positive probability of entering a small neighborhood of the origin rapidly enough so that the connector process does not move far from $r_*$.

To make these statements precise, we set some notation. Let $M_1$ be the constant from \factname \ref{control-list-one} and $m / \sqrt{\eta}$ be the radius of the $N$-sphere $\zgood$. We define $M_2 = (m / \sqrt{\eta}) + M_1$.
For a given tolerance, $\delta_r$, which is set immediately before Equation \eqref{eq:distance-zv-tilde-zv-2}, we define the event  
\begin{equation} \label{eq:defn-omega-1}
\Omega_1 := \left \{ |\rv(1) - r_\ast| \leq \delta_r, \, |\zv(t)| \leq M_2 + 1; \quad \forall \,t \in [0,1] \right \}.
\end{equation}
It is important to note that $M_2$ does not depend on the choice of $\delta_r$.

Taking $t_1 := 2$ and assuming $|R(1) - r_*| < \delta_r$ is sufficiently small, we can show that for any $t_2 > 2$, the event
\begin{equation} \label{eq:defn-omega-2}
\Omega_2 := \big\{ \ |\rv(t) - r_\ast| < \delta/2, |\zv(t)| <
\delta/2; \quad \forall \, t \in [2,t_2] \, \big\}
\end{equation}
has positive probability. The structure of the proof is therefore summarized by:
\begin{equation} \label{eq:top-irred-ineq}
\inf_{t \in [2,t_2]} \tkernel_{t}(x_0, B_{\delta}(x_*)) \geq \pbb_{x_0}\big\{\Omega_2\big\} \geq \pbb_{x_0}\big\{\Omega_2 \, | \, \Omega_1 \big\} \, \pbb_{x_0}\big\{\Omega_1\big\} \geq p_2 p_1
\end{equation}
for some  $p_1 > 0$ and $p_2 > 0$ that are independent of the initial condition $x_0 \in \good$.%

We begin by showing $\inf_{x_0 \in \good} \pbb_{x_0}\{\Omega_1\} \geq p_1$. Let $\tilde
\rv$ be a smooth path in $\rcal$ which was defined in
\eqref{eq:R}.  By \factname \ref{control-list-one} there exists a bounded deterministic control $\tilde \zv$ such that $\tilde \rv = \Psi(r_0,\tilde
\zv)$ over the interval $t \in [0,1]$. The initial value of
the control, $\tilde \zv(0)$, satisfies
$$
|z_0 - \tilde \zv(0)| \leq |z_0| + |\tilde \zv(0)| \leq (m /
\sqrt{\eta}) + M_1
$$
where we recall that $m / \sqrt{\eta}$ is the radius of $\good_z$. In order to apply \factname \ref{control-list-two} we set $M_2 = (m / \sqrt{\eta}) + M_1$ and $T = 1$ while noting that $\tilde \rv(0) = r_0$. Then for a given $\delta_{r} > 0$, there exist positive constants $\gamma_1$ and $\delta_{z,1}$
such that if an instance $\zv$ of the noise satisfies
\begin{equation} \label{eq:distance-zv-tilde-zv}
|\zv(t) - \tilde \zv(t)| \leq M_2 e^{-\gamma_1 t} + \delta_{z,1}, \,
\forall \, t \in [0,1]
\end{equation}
then the corresponding connector process $\rv = \Psi(r_0,
\zv)$ satisfies
$$
|\rv(t) - \tilde \rv(t)| \leq \delta_{r}, \, \forall \, t \in [0,1].
$$
From \factname \ref{control-list-three}, it follows that 
\begin{equation*}
p_1:=\pbb_{z_0} \Big \{\zv : |\zv(t) - \tilde \zv(t)| \leq M_2 e^{-\gamma_1 t} + \delta_{z,1}, \,
\forall \, t \in [0,1]  \Big \}>0
\end{equation*} and $p_1$
does not depend on $z_0$ or $r_0$. We note that by virtue of the proof of \factname \ref{control-list-two} $\delta_{z,1}$ can be chosen to be less than or equal to 1.  Setting $M = M_2 + 1$ we have shown that
$\inf_{x_0 \in \good} \pbb_{x_0}\{\Omega_1\} \geq p_1$. 

Next we prove that $\inf_{x_0 \in \good} \pbb_{x_0}\{\Omega_2 \ | \
\Omega_1 \} > 0$. As mentioned earlier
 we must show that ensuing at time $t=1$, it is
possible to rapidly bring the noise near the origin without
significantly perturbing $\rv$. To this end, we extend the
previous deterministic control $\tilde \zv$ to include the
definition $\tilde \zv(t) = \zerov$ for all $t \in [1, t_2]$. We
also extend the definition of the associated connector so that $\tilde \rv =
\Psi(r_0, \tilde \zv)$ is now well-defined over the full
interval $t \in [0,t_2]$. By hypothesis, $\tilde \rv(1) = r_*$ is
a global minimum of the spring potential and therefore the controlled process
experiences zero forcing from both the controlled noise and the
spring potential. It follows that $\tilde \rv(t) = r_*$ for all $t \in [1,
t_2]$.

We seek to apply \factname \ref{control-list-two} again to show that
$\rv$ remains close to $r_*$ for all $t \in [1,t_2]$. Even though 
$\zv(1)$ is not necessarily close to the control initial value $\tilde \zv(1) =
\zerov$, conditioned on $\Omega_1$, $|\zv(1)| \leq M_2 + 1$.
At this point, we fix the value of $\delta_r > 0$ given in the definition of $\Omega_1$. By \factname \ref{control-list-two}, there exist positive constants $\delta_{z,2} \in  (0,1/2)$,  $\gamma_2 > 0$ and $\delta_{r} > 0$ such that if the connector process
satisfies $|\rv(1) - r_*| \leq \delta_{r}$, and if an instance of the
noise satisfies
\begin{equation} \label{eq:distance-zv-tilde-zv-2}
|\zv(t)| \leq (M_2 + 1) e^{-\gamma_2 (t-1)} + \delta_{z,2}, \ \forall \ t \in
[1,t_2],
\end{equation}
we have $|\rv(t) - r_*| \leq \delta/2, \ \forall \ t \in [1,t_2]$.
Conditioning on $\Omega_1$ and using the Markov property of the
system to shift time values appropriately, \factname \ref{control-list-three} ensures that
the noise satisfies \eqref{eq:distance-zv-tilde-zv-2} with
probability $p_2 > 0$.

It remains to require that $|Z(t)| < \delta / 2$ for all $t \in
[2,t_2]$.  From \eqref{eq:distance-zv-tilde-zv-2}, it suffices to find
a $\gamma_3 \geq \gamma_2$ sufficiently large that
$\exp(-\gamma_3(t-1)) + \delta_{z,2} \leq \delta/2$ for all $t \in
[2,t_2]$.  Indeed, this is the case if we choose $\gamma_3 \geq \ln\big(\frac{\delta}{2} - \delta_{z,2}\big)^{-1}$
and we are done.
\end{proof}

In order to complete the proof of the Uniformly Accessible Neighborhood Condition of Lemma \ref{thm:two-conditions} we need to extend Lemma \ref{lem:top-irred} to apply to all initial conditions in $\cntr$. To do this, we need the particle separation property from Lemma~\ref{lem:nonlinear-near-zero}.
\begin{lemma}[Topological irreducibility on $\cntr$] \label{lem:top-irred-C}
Given
  a $\delta > 0$, there exists a $t_1' >0$ so that for any $t\geq t_1'$
  there is an $\alpha_1' > 0$ with
  \begin{align*}
    \inf_{x_0 \in \cntr} \tkernel_{t}(x_0, B_{\delta}(x_*)) \geq \alpha_1'
  \end{align*}
\end{lemma}
\begin{proof}
  Set $t_1'=t_1 +1$ where $t_1$ is the constant from
  Lemma~\ref{lem:top-irred} and let $\tau~:=~\inf\{t > 0 \ : \ (R(t),Z(t)) \in \good\}$. Now for any $t\geq t_1'$ and fixed $x_0 \in \bad$ we have
  \begin{align*}
    \cP_t(x_0,B_\delta(x_*)) &\geq \Big(\mathbb{P}_{x_0}\!\left\{ X_t \in B_{\delta}(x_*) | \tau \leq 1\right\}\Big) \Big(\mathbb{P}_{x_0}\{\tau \leq 1\}\Big)\\
	&\geq \Big( \inf_{x
      \in \good} \inf_{s\in[0,1]} \tkernel_{t-s}(x, B_{\delta}(x_*))
    \Big) \Big( \mathbb{P}_{x_0}\{ \tau \leq 1\} \Big) \geq \alpha_1 \mathbb{P}_{x_0}\{ \tau \leq 1\} 
  \end{align*}
  where $\alpha_1$ is from Lemma~\ref{lem:top-irred}. Finally, we take the $\inf$ over all initial conditions $x_0 \in \bad$. Applying Lemma~\ref{lem:nonlinear-near-zero} with $m = \psi(\phi(\rhogood))$ and $M = m/\sqrt{\eta}$, we conclude there exists an $\alpha > 0$ such that
\begin{equation*}
	\inf_{x_0 \in \bad}\cP_t(x_0,B_\delta(x_*)) \geq \alpha_1 \inf_{x_0 \in \bad} \mathbb{P}_{x_0}\{ \tau \leq 1\} \geq \alpha_1 \alpha > 0.
\end{equation*}
Setting $\alpha_1'=\alpha\alpha_1$ completes the proof.
\end{proof}

\subsection{Measure Theoretic Irreducibility via H\"ormander's Condition}\label{sec:Hormander}
\begin{lemma}[Absolute continuity of the transition density] \label{lem:meas-irr}
Let $\{X(t) =
(\rv(t),\zv(t) )\}_{t\geq 0}$ be a Markov process with transition kernel
$\cP_t(x, U)$. Then for any $t > 0$, there exists a smooth function $p_t(x, y)$, such that
\begin{equation*}
  \cP_t(x, U) = \int_U p_t(x,y) dy
\end{equation*}
for every $U \in \borel{\cntr}$, where
$p_t(x,y)$ is jointly continuous in $(x,y) \in \cntr \times \cntr$.
\end{lemma}
\begin{remark}
  In fact, the system has a density for all $(x,y) \in
  \xbb \times \xbb$. However, due to the periodicity of our
  forcing, proving this would require an additional small argument. Since we
  do not need this fact, we refrain.
\end{remark}
\begin{proof}
  The claim follows from a now classical theorem of H\"ormander which
  states that if a diffusion on an open manifold satisfies a certain
  algebraic condition then $L_1=\partial_t -\gen$ and $L_2=\partial_t
  -\gen^*$ are both hypoelliptic in $\cntr$ where $\gen$ is the generator of the
  diffusion $\Xv(t)$ and $\gen^*$ is its adjoint. A combination of
  It\^o's formula and the fact that we have shown that the
  singularities of the potential are unattainable demonstrates that
  $L_1u =0$ and $L_2u =0$ have distribution-valued solutions.
  Hypoellipticity of the operators ensures first that these distribution-valued
  solutions are in fact smooth functions.  Furthermore, hypoellipticity implies the existence of
  fundamental solution, which in turn yields continuity in the second variable throughout
  the center of the space $\cntr$.

  The fact that the density is jointly
  continuous follows after a little more work. The argument is laid
  out in its entirety for $\rbb^N$ valued diffusions in Section 7.4 of
  \cite{stroock08PDE}. In particular, see Theorem~7.4.3 and Theorem~7.4.20.
  Essentially, the same proofs follow in our setting since we have
  shown the system is a well defined diffusion on the manifold $\xbb$ with
  distribution-valued solution. Hypoellipticity and the properties which follow
  are local statements, and therefore still apply. The needed results in the general setting, as
  opposed to $\rbb^N$, can be found in Chapter~22 of
  \cite{hormander85III}, noting in particular
  Theorem 22.2.1. However, the
  presentation in \cite{stroock08PDE} is closer to the exact statements we
  need.

  We now turn to the explicit calculations needed to show that
  H\"ormander's condition is satisfied. We recast the
  system of equations \eqref{eq:defn-z} and \eqref{eq:defn-r} as
  \begin{equation*}
    d \Xv(t) = A(\Xv(t)) \,dt+ B d\Wv(t)
  \end{equation*}
where $A(x) \in \rbb^{2+N}$ and $B\in \rbb^{(2+N) \times (2+N)}$ with
\begin{align*}
  A(x) =
  \begin{pmatrix}
    - \nabla \Phi(r) + S(r) z \\
    - \length^2\nu |\kv|^2 \zv
  \end{pmatrix}\,, \quad
  B =
  \begin{pmatrix}
    0 & 0 \\
    0 & \tilde B
  \end{pmatrix}\;.
\end{align*}
where $\tilde B$ is an $N \times N$ diagonal matrix with diagonal
entries $\sqrt{2 \beta \nu} \lambda \sigma_\kv$.
 In this notation, the generator $\gen$ of the
diffusion is given in terms of a test function $\f$ by
\begin{align*}
  (\gen\f)(x) = (A \cdot \nabla)\f(x) +\frac12 \sum_{\kv \in
    \modeset}(B_{\kv} \cdot \nabla)^2\f(x)
\end{align*}
where $B_\kv$ is the column of $B$ associated with the mode
direction $\kv\in \modeset$.

For two vector fields $A,B$ let $[A,B] := AB - BA$
denote their the commutator or Lie bracket.
In our simplified setting
where $B_\kv$ is a constant vector-field one has
\begin{align*}
  \left[A(\xv), B_\kv\right] = \frac{\partial}{\partial z_\kv} A(\xv)
  =
  \begin{pmatrix}
    \sin(\length\kv\cdot \rv) \frac\kvperp{|\kv|} \\ - \length^2
    \nu\kvsq\, \ev_\kv
  \end{pmatrix}
\end{align*}
where $\ev_\kv$ is the is the unit basis vector in $\rbb^N
=\rbb^{|\modeset|}$ associated to the mode direction $\kv \in
\modeset$. Moreover all the iterated Lie brackets of
$B_\kv$ and $A(x)$ are $0$. Thus to satisfy the
H\"ormander's condition at the point
$x$, it is required that
\begin{align*}
  \text{span}\Big\{
  B_\kv,[A(x),B_\kv] : \kv \in \modeset\Big\}=\rbb^{2+N}.
\end{align*}
The set $\{[A(x),B_\kv]\}_{\kv \in \modeset}$ will span
$\rbb^{2+N}$ if and only if the set $\{\sin(\kv\cdot r)
\kvperp\}_{\kv \in \modeset}$ spans $\rbb^2$ since the set $\{ \ev_\kv :
\kv \in \modeset\}$ spans $\rbb^N$. We recall that by assumption
$\modeset$ contains at least three pairwise independent vectors
which we label $\kv_1,\kv_2$, and $\kv_3$. One may note that due to
the periodicity of the forcing, $\sin(\length\kv \cdot r)=0$ for
all $r \in L \zbb^2$.  Taking $L \gg \rho_0^2$ will ensure that all of
these points lie outside of $\cntr$. Thus restricting to $x \in
\cntr$ at least two of $r \cdot \kv_i$ are nonzero and the lemma is
proved.
\end{proof}

\subsection{Ergodicity of generalizations}
\label{sec:generalizations}

In the derivation of the model equations \eqref{eq:defn-z}
and \eqref{eq:defn-r} we imposed the simplifying assumption that the
center of mass $\mv(t) := \frac{1}{2} (\xv_1(t) + \xv_2(t))$ is held
at zero (see Appendix).  This greatly simplified the presentation
and did not affect the conclusion that the bead-spring system has an
ergodic connector process $\rv(t)$.  Indeed the fluid velocity term
with nonzero $\mv(t)$ is given by Eq. \eqref{eq:nonzero-mv}:
\begin{align*}
& \frac{1}{2}[\uv(\xv_1(t),t) - \uv(\xv_2(t),t)]\\
& \nonumber \qquad = \sum_{\kv \in \modeset} \left[\cos(\length
\kv\cdot\mv) Z_\kv - \sin(\length \kv\cdot\mv) Y_\kv
\right]\sin(\length\kv \cdot \rv) \frac{\kvperp}{|\kv|}
\end{align*}
where the $\{Y_\kv\}$ are a second set of OU-processes defined
exactly as the $\{Z_\kv\}$.

Because the $\mv$ terms appear inside of cosines and sines, there is
no new significant contribution to the Lyapunov function
calculation.  For the H\"ormander condition, the additional terms in
the coefficients of the noise introduce more ``dead spots'' in the
forcing, but still one needs only \emph{four} pairwise linearly
independent vectors $\kv_i$ in the mode set $\modeset$ to ensure
that at least two of the vectors
$$
\big\{\left[\cos(\length \kv_i \cdot M) - \sin(\length \kv_i \cdot
M)\right] \sin(\length \kv_i \cdot R)\kv_i^\perp \big\}
$$
are nonzero.  This guarantees the existence of a continuous
transition density and it remains to show the $\delta$-ball
controllability as in Lemma \ref{lem:top-irred}.  While the
calculation is more involved, the principle of identifying the
region of good control $\good$, where the coefficients of the
$\rv$-differential equation are uniform, still applies. Furthermore,
since the differential equation for $\rv$ is linear in the
$\{Y_\kv\}$ and $\{Z_\kv\}$, we may still solve for stochastic
control explicitly in terms of the desired path $\Gamma$ as long as
the new Stokes matrix is non-degenerate.  Again, this is guaranteed
by the hypothesis that $\modeset$ contains at least four pairwise
linearly independent vectors.

\bigskip

\section{Acknowledgements}

The authors would like to thank Peter March, Greg Forest, Andrew Stuart and Davar Khoshnevisan for helpful conversations and suggestions during the
development of this paper. We would also like to thank the anonymous referees for their careful reading and thoughtful comments. All three authors wish to thank the NSF for its support through grant DMS-0449910,
DMS-0854879 and DMS-1107070.

\appendix

\section{Derivation of the model} \label{sec:model}

In the overdamped, highly viscous regime, it is reasonable to
neglect the nonlinear term in Navier-Stokes equations
\cite{1989-ottinger-rabin}. Following \cite{1986-walsh},
\cite{1992-daprato-zab}, \cite{1999-dalang} and \cite{2006-mckinley}
we have the stochastic PDE given in Section 1, Eq.
\ref{eq:defn-stokes},
\begin{align*}
\partial_t \fluidv(x,t) - \viscosity\Delta
\fluidv(x,t) + \nabla p(x,t) &= F(dx, dt), \quad
\nabla \cdot \fluidv(x,dt) = 0
\end{align*}
with periodic boundary conditions on the rectangle $ [0,L]
\times [0,L]$ where $L$ is presumed to be very large. For this development (see also \cite{2002-sigur-inertial}) we assume that the space-time forcing is a mean zero complex-valued Gaussian process with covariance
\begin{equation*}
\E{F^\alpha(x,t) \overline{F^\beta(y,s)}} = (t \wedge s) 2 k_B T \nu
\delta_{\alpha \beta} \Gamma(x - y)
\end{equation*}
where $\alpha, \beta \in \{1, 2\}$ and $\delta_{\alpha \beta}$
is a Kronecker delta function. It follows that
\begin{equation*}
F(x,t) = \frac{\sqrt{2 k_B T \nu}}{L} \sum_{\kv \in \zbb^2 \setminus \zerov} e^{\length i \kv \cdot x} \sigma_{\kv} B_{\kv}(t) \;
\end{equation*}
where $\{B_k\}$ is a collection of complex-valued 2-d Brownian motions and the coefficients $\sigma_\kv$ are related to the spatial correlation function $\Gamma$ through the Fourier relation $
\Gamma(x) = \frac{2}{L^2} \sum_{k \in \zbb^2 \setminus \{\zerov\}}
e^{\lambda i \kv \cdot x} \sigma_{\kv}^2$.
In order to construct a real-valued noise of the form \eqref{eq:defn-noise}, one can set $\sigma_{-k} = \sigma_k$ and $B_{-k} = \overline{B_k}$ and for all $k$.

To compute the Fourier transform of the SPDE, we note that the transform of the noise is given by
\begin{align*}
  \int_{[0,L]^2} e^{-\length i \kv \cdot x} F(x,t) dx &= \int_{[0,L]^2}
  e^{-\length i \kv \cdot x} \frac{\sqrt{2 k_B T \nu}}{L} \sum_{\jv \in \zbb^2 \setminus \zerov}
  e^{\length i \jv \cdot x} \sigma_{\jv}
  B_{\jv}(t) dx \\
  &= \frac{\sqrt{2 k_B T \nu}}{L} \sum_{\jv \in \zbb^2 \setminus \zerov} \sigma_{\jv}
  B_{\jv}(t) \int_{ [0,L]^2} e^{-\length i (\kv - \jv) \cdot x} dx \\
  &= \sqrt{2 k_B T \nu} L \sum_{\jv \in \zbb^2 \setminus \zerov}  \sigma_{\jv}
  B_{\jv}(t) \delta_{\kv \jv} = \sqrt{2 k_B T \nu} \,L \,\sigma_{\kv} \,B_{\kv}(t).
\end{align*}
The SPDE transforms into the infinite dimensional system
\begin{align}
\label{eq:stokes-hat-a} d \fluidvhat_{\kv}(t) + \length^2 \viscosity
|\kv|^2 \fluidvhat_{\kv}(t) + \length i \kv \hat p_{\kv}(t)&=
\sqrt{2 k_B T \nu} L \sigma_{\kv} d B_{\kv}(t), \\
\label{eq:stokes-hat-b} \length i \kv \cdot \fluidvhat_k(t) &= \zerov.
\end{align}
For the sake of completing the formal argument, suppose for the
moment that the forcing term is smooth with derivative $f$. By
taking the dot product of $\kv$ with the terms of equation
(\ref{eq:stokes-hat-a}), the first two terms vanish -- via
incompressibility condition (\ref{eq:stokes-hat-b}) -- leaving the
identity
\begin{equation}
\length i |\kv|^2 \hat p_{\kv}(t) = \sqrt{2 k_B T \nu} L
\sigma_{\kv} \kv \cdot f(t).
\end{equation}
Substituting back into (\ref{eq:stokes-hat-a}) and gathering
$f(t)$ terms on the right-hand side yields
\begin{equation}
  \label{eq:stokes-hat-proj-one} d \fluidvhat_{\kv}(t) +
  \length^2  \viscosity |\kv|^2 \fluidvhat_{\kv}(t) = \sqrt{2 k_B T \nu} L \sigma_{\kv}
  \Big(f(t) - \frac{\kv \cdot f(t)}{|\kv|^2} \kv \Big).
\end{equation}
The projection on the right hand side has two standard
representations:
$$
f - \frac{\kv \cdot f}{|f|^2} \kv = \Big(I - \frac{k
    \otimes \kv}{|\kv|^2}\Big) f = \frac{f \cdot
  \kvperp}{|\kv|^2} \kvperp,
$$
where $\kvperp := \vect{-k_2}{k_1}$. Applying Duhamel's principle
and assuming initial condition is taken from the stationary distribution, we have the following representation for solutions to the fluid mode equations
\begin{align*}
\fluidvhat_\kv(t) &= e^{-\length^2 \viscosity \kvsq t} \fluidvhat_\kv(0) + \sqrt{2 k_B T \nu} \sigma_{\kv} L \int_0^t e^{-
\length^2 \viscosity \kvsq (t-s)} \projop d B_\kv(t) \\
&= \projop \zeta_{\kv}(t)
\end{align*}
where we define $\zeta_{\kv}$ to be the appropriate complex valued 2-d Ornstein-Uhlenbeck process,
$$
d \zeta_{\kv}(t) = -\lambda^2 \nu |\kv|^2 \zeta_{\kv}(t) dt + \sqrt{2
k_B T \nu} L \sigma_\kv d B_{\kv}(t)
$$
with $\zeta_{\kv}(0)$ normally distributed according to the respective stationary distributions for each $\kv$.
We therefore have the solution for the fluid velocity field,
\begin{align*}
\uv(x,t) &= \frac{1}{L^2} \sum_{\kv \in \zbb^2 \setminus
\zerov} e^{\lambda i \kv \cdot x} \projop \zeta_\kv = \frac{1}{L^2} \sum_{\kv \in \zbb^2 \setminus \zerov} e^{\lambda i
\kv \cdot x} \frac{\zeta_{\kv} \cdot \kvperp}{|\kv|^2} \kvperp.
\end{align*}
After defining $\xi_{\kv} := \frac{1}{L^2} \frac{\zeta_{\kv} \cdot
\kvperp}{|\kv|}$, we have the complex valued 1-$d$ OU processes that drive the
dynamics
$$
d \xi_\kv(t) = - \frac{4 \pi^2 \nu |\kv|^2}{L^2} \xi_\kv(t) dt +
\frac{\sqrt{2 k_B T \nu} \sigma_{\kv}}{L} d W_\kv(t)
$$
Imposing the condition that we require real-valued solutions, after
Fourier inversion we have the following trigonometric expansion for
2-$d$ stochastic Stokes
\begin{equation} \label{eq:fluid-defn}
\uv(x,t) = \sum_{\kv \in \zbb^2 \setminus \zerov} \left(\cos
(\length \, \kv \cdot x) Y_\kv + \sin(\length \, \kv \cdot x)
Z_\kv \right) \frac{\kvperp}{|\kv|}.
\end{equation}
where the $Y_\kv$ and $Z_\kv$ are the real and imaginary parts of $\xi$ respectively.

In this paper, we study the dynamics of the two beads in normal
coordinates: $\mv(t) = \frac{1}{2} (\xv_1(t) + \xv_2(t))$ and
$R(t) = \frac{1}{2} (\xv_1(t) - \xv_2(t))$,
\begin{align*}
\ddt \mv(t) &= \frac{1}{2} [\uv(\xv_1(t),t) + \uv(\xv_2(t),t)] \\
\ddt \rv(t)  &= - \nabla \Phi(\rv(t)) +
\frac{1}{2}[\uv(\xv_1(t),t) - \uv(\xv_2(t),t)].
\end{align*}
In light of equation \eqref{eq:defn-stokes}, we may write the radial
process and the noise together as a Markovian system of SDE with two
degenerate directions. In order to write the system in this form, we
first record the identity
\begin{align}
& \label{eq:nonzero-mv} \frac{1}{2}[\uv(\xv_1(t),t) - \uv(\xv_2(t),t)]\\
& \nonumber \qquad = \sum_{\kv \in \modeset} \left[\cos(\length
\kv\cdot\mv(t)) z_\kv(t) - \sin(\length \kv\cdot\mv(t)) y_\kv(t)
\right]\sin(\length\kv \cdot \rv(t)) \frac{\kvperp}{|\kv|}.
\end{align}
For the majority of the paper, we used the simplification $\mv(t) =
0$ for all $t$. This does not have any effect on the ergodic results
as is discussed in Section \ref{sec:generalizations}, but it does
significantly streamline the presentation. Altogether we have the
definition of the dynamics given in Section 1, Eq.~\eqref{eq:defn-r}.

\subsection{Degeneracy when there is no repulsive force}
\label{sec:nearZero}

Putting aside existence and uniqueness for a moment, we make a quick
calculation that reveals a degeneracy for the bead-spring model with a Hookean or FENE
spring potential with truncated stochastic Stokes forcing.  Namely, under mild
conditions, when the two beads come close together, the fluid velocity vectors they respectively see will become so correlated, the beads will never separate.

\begin{prop}[Degeneracy of the non-repulsive case]\label{thm:hookean}
Let $\rv$ and the family $\{Z_\kv\}_{\kv \in \modeset}$ satisfy the
system of differential equations \eqref{eq:defn-z} and
\eqref{eq:defn-r}. Let the spring potential be given by $\Phi(r) = \frac{\gamma}{2} |r|^2$ or $\Phi(r) = \Phifene(r)$ as defined by \eqref{eq:defn-fene}. Then there exists a $\gamma_0$ so that if $\gamma > \gamma_0$ then
$$
\lim_{t \to \infty} \rv(t) = \zerov
$$
almost surely.
\end{prop}

\begin{proof} 
We first note that for all $r$ satisfying $|r| \in (0, \rhomax)$
\begin{align*}
	\nabla \Phifene(r) \cdot r &= \frac{\gamma |r|^2}{1 - |r|^2/\rhomax^2} \geq \gamma |r|^2.
\end{align*}
It follows that both the Hookean and FENE potential cases, the process $|\rv(t)|^2$ satisfies the following pathwise ODE bound,
\begin{align*}
\ddt |\rv(t)|^2 &= - 2 \nabla \Phi(\rv(t)) \cdot \rv(t) + 2\sum_{\kv \in \modeset} \sin\!\left(\length \kv \cdot \rv(t)\right) \frac{\kvperp \cdot \rv(t)}{|\kv|} Z_\kv(t) \\
&\leq - 2 \gamma |\rv(t)|^2 + 2 \length \sum_{\kv \in \modeset} |\kv \cdot \rv(t)| |\kvperp \cdot \rv(t)| \frac{|Z_\kv(t)|}{|\kv|} \\
&\leq - 2 \gamma |\rv(t)|^2 + 2 \length |\rv(t)|^2 \|Z(t)\|_1
\end{align*}
where $\|Z(t)\|_1 := \sum_{\kv \in \modeset} |\kv| |Z_\kv(t)|$.

This differential inequality implies
\begin{equation} \label{eq:ineq-degenerate-hookean}
|\rv(t)|^2 \leq |R(0)| \exp\left[- 2\gamma t + 2\length
  \int_0^t \|Z(s)\|_1 ds\right] \;.
\end{equation}
Recall that in its stationary distribution, the law of each $Z_k(t)$ is normal with mean zero and variance $\beta \sigma_k^2 / |k|^2$ and therefore $\E{|Z_k|} = \sqrt{\frac{2 \beta}{\pi}} \frac{\sigma_k}{|k|}$. By the Law of Large Numbers
\begin{equation}
	\lim_{t \to \infty} \frac{1}{t} \int_0^t |Z_k(s)| ds = \sqrt{\frac{2\beta}{\pi}} \, \frac{\sigma_k}{|k|}
\end{equation} 
almost surely and so
$$
\lim_{t \to \infty} \frac{1}{t}\int_0^t \|Z(s)\|_1 ds
= \sqrt{\frac{2 \beta}{\pi}} \sum_{\kv \in \modeset} \sigma_k
$$
almost surely. Since we are only considering a finite number of modes, the above sum is finite.  Therefore, if $\gamma > \gamma_0 := \lambda\sqrt{\frac{2 \beta}{\pi}}\sum_k \sigma_k$, then $|\rv(t)|^2 \rightarrow 0$
almost surely as $t\rightarrow \infty$.

\end{proof}

\subsection{A note on the mollifier function $\psi$}
\label{sec:psi}

Recall the mollifier function $\psi$ that appeared in the Lyapunov function \eqref{eq:defn-lyapunov} and in the global estimate in the existence and uniqueness Proposition \ref{prop:exist-unique},
\begin{equation*}
\psi(x) :=  \left\{\begin{array}{cl} 0, & 0 \leq x \leq a,\\
        c \, (x - a) \exp \big(\frac{-1}{(x-a)^2}\big), & x > a. \end{array}\right.
\end{equation*}
where $a = \Phi(\rho_0)$. Since $\lim_{x \to 0} x^{\alpha}e^{-1/x^2} = 0$ for any $\alpha \in \rbb$, it follows that for any $n \in \mathbb{N}$, the $n$-th derivative of $\psi$ satisfies $\lim_{x \to a}\psi^{(n)}(x) = 0$. Therefore $\psi$ and all of its derivatives are continuous for all $x \in \rbb_+$. Furthermore, we have the following proposition.

\begin{prop} \label{prop:psi-bounds} There exists a constant $C > 0$ such that
	\begin{equation} \label{eq:psi-bounds}
		\psi(x) \leq x \psi'(x) \leq \psi(x) + C
	\end{equation}
	for all $x \in \rbb_+$.
	Furthermore, $\|\moll'\|_\infty < \infty$
\end{prop}
\begin{proof}
	This is trivially true for all $x \in [0,a]$, since $\psi(x) = x \psi'(x) = 0$ for all $x$ in this range. For $x > a$, we compute that $x \psi'(x) = \psi(x) + r(x)$ where the remainder term is given by $r(x) = c \left(a + 2 x (x-a)^{-2}\right) \exp\!\left(-(x-a)^{-2}\right)$.  This remainder term is always positive, is continuous for all $x > a$ and satisfies $\lim_{x \to a} r(x) = 0$ and $\lim_{x \to \infty} r(x) = a$.  It follows that there exists a $C > 0$ for all $x \geq a$ we have $0 \leq r(x) \leq C$.  The inequalities \eqref{eq:psi-bounds} follow.
\end{proof}

\subsection{Existence, uniqueness of the bead-spring model}
\label{sec:exist-unique}

We confirm the global existence and uniqueness of the bead-spring
model proposed by Equations \eqref{eq:defn-z} and \eqref{eq:defn-r}. 
Since we assume that $|\modeset| = N \in \mathbb{N}$ throughout the main part of this paper, we retain that assumption here.

\begin{prop} \label{prop:exist-unique}
Suppose that the spring potential $\Phi$ satisfies Assumption
\ref{a:potential-singularity}. Let
$\{Z_\kv(t):t\geq0\}_{\kv \in \modeset}$ be a solution to the family
of SDEs \eqref{eq:defn-z} with initial conditions $Z_k(0) = z_k \in \rbb$ for all $k \in \modeset$.
Then, almost surely, there exists a unique global solution to the 2-dimensional ODE
\begin{equation} \label{eq:defn-r-appendix}
\ddt \rv(t)  = - \nabla \Phi(\rv(t)) + \sum_{\kv \in \modeset} \sin\!\left(\length \, \kv \cdot \rv(t)\right) \frac{\kvperp}{|\kv|} Z_\kv(t)
\end{equation}
with the initial condition $\rv(0) = r_0 \in \dcal \setminus \{0\}$. 
\end{prop}

\begin{proof}
Let $\ep > 0$ be given and define the stopping stopping time 
$\tau_\ep := \inf\{t > 0 \, : \, |R(t)| < \ep \text{ or } \moll(\Phi(R(t))) > \ep^{-1}\}$
where $\moll$ is the function defined in the previous section.
We will first prove there exists a unique stopped solution $R(t \wedge \tau_\ep)$ to \eqref{eq:defn-r-appendix}.  Subsequently we show that $\sup \{\tau_\ep\} = \infty$ almost surely.

We rewrite \eqref{eq:defn-r-appendix} in terms of the Stokes matrix defined by \eqref{eq:defn-stokes-matrix},
\begin{equation} \label{eq:r-defn-concise}
	\ddt \rv(t) = - \nabla \Phi(\rv(t)) + S(\rv(t))\zv(t).
\end{equation}
In order to apply the standard Picard-Lindel\"of Theorem (see for
example, \cite{book-hale}), we think of the vector $Z(t) = (Z_0(t), Z_1(t), \ldots, Z_N(t))$ 
as a time-inhomogeneous coefficient. To prove that there exists a unique local solution to
\eqref{eq:defn-r-appendix} it is sufficient to show that the functions $\nabla
\Phi(r)$ and $S(r) \zv(t)$ are continuous in $\dcal \times \rbb_+ \setminus \{0 \times \rbb_+\}$ and locally Lipschitz in the variable $r$.  By Assumption \ref{a:potential-singularity}, this condition is satisfied by $\nabla \Phi(r)$. For the last term in \eqref{eq:r-defn-concise}, given an instance of
$\zv$, we have
\begin{align*}
	|S(r_1)\zv(t) - S(r_2)\zv(t)| &\leq \sum_{\kv \in \modeset} |\sin(\lambda \kv \cdot r_1) - \sin(\lambda \kv \cdot r_2)| |Z_\kv(t)| \\
	&\leq \lambda |r_1 - r_2| \|Z(t)\|_1
\end{align*}
where we recall $\|Z(t)\|_1 := \sum_{k \in \modeset} |k| |Z_\kv(t)|$.
The function $S(r)\zv(t)$ is continuous in $t$ almost surely since $|S(r)\zv(t_1) -
S(r)\zv(t_2)| \leq \|S(r)\|_F |\zv(t_1)-\zv(t_2)|$ and the vector OU process $Z(t)$
is continuous almost surely.%

We now show that the process cannot blow up to $\rhomax$ in finite time.  To this end we consider the process $\moll(\Phi(R(t)))$ which is constant inside a radius of size $\rho_0$ but then grows to infinity with the potential function as $|R|$ tends to $\rhomax$. By showing $\moll(\Phi(R(t)))$ is bounded above by a 1-d linear ODE, this suffices to show global existence and uniqueness.  %
For a given instance of the noise $Z(t)$, we have
\begin{equation*} 
	\ddt \moll(\Phi(\rv(t))) = \moll'(\Phi(R(t))) \left(- |\nabla \Phi(R(t))|^2 + \nabla \Phi(R(t)) \cdot [S(R(t)) Z(t)]\right)
\end{equation*}
For given values $r \in \rbb^2$ and $z \in \rbb^N$ we bound the Stokes forcing term by applying Young's inequality followed by the matrix form of Cauchy-Schwarz:
\begin{align*}
\nabla \Phi(r) \cdot (S(r) z) &\leq  \frac12 |\nabla \Phi(r)|^2 + \frac{1}{2} |S(r) z|^2 \leq  \frac12 |\nabla \Phi(r)|^2 + \frac12 \|S(r)\|_F^2 |z|^2\\& \leq  \frac12 |\nabla \Phi(r)|^2 + \frac12 N^2 |z|^2.
\end{align*}
The inequality $\|S(r)\|_F \leq N$ is given in the proof of Lemma \ref{lem:nonlinear-near-zero}. 

To estimate the first term of the mollified ODE, we consider two cases: (i) $|r| \leq \rho_0$ and (ii) $|r| > \rho_0$. In case (i), $\moll'(\Phi(r)) = 0$ and the entire term disappears.
Trivially, $ - \moll'(\Phi(r)) |\nabla \Phi(r)|^2 = 0 = - \gamma \moll(\Phi(r)).$

For case (ii), we employ the spring potential assumption \eqref{eq:spring-large-r} that for some $\gamma > 0$ if $|r| > \rho_0$ then $|\nabla \Phi(r)|^2 \geq \gamma \Phi(r)$. Furthermore, by Proposition \ref{prop:psi-bounds}, the mollifier $\moll$ satisfies $ \moll'(\Phi(r)) \Phi(r) \geq \moll(\Phi(r))$. We obtain
\begin{align} \label{eq:psi-phi-ineq}
- \moll'(\Phi(r)) |\nabla \Phi(r)|^2 \leq - \gamma \moll'(\Phi(r)) \Phi(r) \leq - \gamma \moll(\Phi(r)).
\end{align}
Altogether, we have the differential inequality
\begin{equation}
	\ddt \moll(\Phi(\rv(t))) \leq - \frac{\gamma}{2} \moll(\Phi(R(t))) + \frac{N^2}{2} \|\moll'\|_\infty |Z(t)|^2
\end{equation}
Define $Y(t)$ to be the solution to the linear ODE
\begin{equation*}
	\ddt Y(t) = - \frac{\gamma}{2} Y(t) + \frac{N^2}{2} \|\moll'\|_\infty |Z(t)|^2
\end{equation*}
with $Y(0) = \moll(\Phi(R(0))$. By definition, $\moll(\Phi(R(t)) \leq Y(t)$.  By virtue of the fact that the forcing term is positive, $Y(t) > 0$ for all $t$ and, defining $\tau_M = \inf\{t > 0 : Y(t) > M\}$, standard properties of linear ODEs and global existence of the $N$-dimensional Ornstein-Uhlenbeck imply that $\sup_{M > 0} \tau_M =\infty$.

We now show that $\sup_{\ep > 0} \tau_\ep = \infty$ almost surely by demonstrating that the $R$-dynamics do not hit zero in finite time.  The idea here is that for the connecter process to hit zero, the noise must blow up in finite time and this is not possible since our noise is bounded on any finite time interval. Indeed, by Assumption \ref{a:potential-singularity}, there exists  $\ep_0> 0$ such that $-\nabla \Phi(r)\cdot r \geq \gamma_0 > 0$ for all $r$ with $|r|< \ep_0$. Suppose $R(T) = 0$ for some $T \in \mathbb{R}_+$. From
 the above discussion and Equation \eqref{eq:defn-r-appendix} it follows
 that $\frac{d}{dt} |R(t)|^2$ is almost surely continuous.  Thus $\frac{d}{dt} |R(t)|^2 < 0$
 in a subinterval of the set $[T-\delta, T]$ for some $\delta > 0$.  Without loss of generality, we may assume that $|R(t)| < \epsilon_0$ for $[T-\delta, T]$. Let $M := \sup_{t\in[T-\delta, T]} \|Z(t)\|_1$. In this regime, we have
\begin{align*}
	\ddt |\rv(t)|^2 &= - 2 \nabla \Phi(\rv(t)) \cdot \rv(t) + 2\sum_{\kv \in \modeset} \sin\!\left(\length \kv \cdot \rv(t)\right) \frac{\kvperp \cdot \rv(t)}{|\kv|} Z_\kv(t) \\
	&\geq 2 \gamma_0 - 2 \length \sum_{\kv \in \modeset} |\kv \cdot
        \rv(t)| |\kvperp \cdot \rv(t)| \frac{|Z_\kv(t)|}{|\kv|}  \geq 2 \gamma_0 - 2 \length M N |\rv(t)|^2.
\end{align*}
However, the right-hand side is positive when $|R(t)|^2 \leq \gamma_0 / (\length M N)$, contradicting the hypothesis that $\ddt |R(t)|^2 < 0$ in a subinterval of $[T-\delta,T]$ when
$|R(t)|$ is small enough. Therefore the origin is unattainable in finite time.
\end{proof}

\subsection{The Lyapunov function}\label{sec:Lyap}

The proof for the Lyapunov estimate, Lemma \ref{lem:lyapunov}, proceeds similarly to the proof of the upper bound in the Existence and Uniqueness Proposition \ref{prop:exist-unique}. The only differences arise from the need to treat the $R(t)$ and $Z(t)$ dynamics simultaneously. For the sake of easy reference, we recall the definition of the Lyapunov function $V(r,z) = \moll(\Phi(r)) + \eta |z|^2$ where $\moll$ is defined in Section \ref{sec:psi} and $\eta$ is to be defined in the following proof.

\begin{proof}[Proof of Lemma \ref{lem:lyapunov}]
The generator $\gen$ for the Markov process $\Xv(t) :=
(\rv(t),\zv(t))$ is given by
$$
\gen := \left(-\nabla \Phi(r) + S(r) z \right) \cdot \nabla_{r} + \nu \lambda^2 \Big(\sum_{\kv \in \modeset} - \kvsq z_\kv \frac{\partial}{\partial z_{\kv}} + \beta \sigma^2_\kv
\frac{\partial^2}{\partial z_{\kv}^2}\Big)
$$
It suffices to find an $a > 0$ and $b > 0$ such that
\begin{equation} \label{eq:gen-v-ineq}
\gen V(x) \leq - a V(x) + b.
\end{equation}
From \eqref{eq:gen-v-ineq}, using Ito's formula and Gronwall's
inequality one can show that
$
(\cP_t V)(x) \leq e^{-at} V(x) + b/a
$.
Thus we have $c_0 = e^{-at}$ with $c_1 = b/a$.  The restriction on the constants $(c_0, c_1, \rhogood)$ from Theorem \ref{thm:HM-harris} (in light of of the definition of $\cntr$ in Equation \ref{eq:defn-center}) translates to the following constraint on $(a,b,\rhogood)$:
\begin{equation} \label{eq:b-restriction}
	b < \frac{1}{2} a \psi(\phi(\rhogood)) (1 - e^{-at}).
\end{equation}
Applying $\gen$ to the Lyapunov function $V$ yields:
\begin{align*}
\gen V(r,z) &= \moll'(\Phi(r)) \left(-|\nabla \Phi(r)|^2 + (S(r) z) \cdot \nabla \Phi(r)\right) \\ 
& \qquad + 2 \eta \nu \lambda^2 \sum_{\kv \in \modeset} \left(- |k|^2 z_\kv^2 + \beta \sigma_\kv^2 \right).
\end{align*}
In bounding the Stokes forcing term we must make a slightly sharper estimate than the one used in the proof of Proposition \ref{prop:exist-unique}. We apply Young's inequality (with $\delta \in (0,1)$ to be chosen below) followed by the matrix form of Cauchy-Schwarz and the inequality $\|S(r)\|_F \leq N$ which is given in the proof of Lemma \ref{lem:nonlinear-near-zero}:
\begin{align*}
(S(r) z) \cdot \nabla \Phi(r) &\leq  \frac{1}{4 \delta} |S(r) z|^2 + \delta |\nabla \Phi(r)|^2 
\leq  \frac{1}{4\delta} N^2 |z|^2 +  \delta |\nabla \Phi(r)|^2.
\end{align*}
Denoting $\kmin := \min_{\kv \in \modeset} \{|\kv|\}$ and
$\|\sigma\|_0^2 = \sum_{\kv \in \modeset} \sigma_{\kv}^2$, after collecting terms we have
\begin{align}
\label{eq:gen-v-collected} \gen V(x) &\leq - (1 - \delta) \moll'(\Phi(r)) |\nabla \Phi(r)|^2 + 2 \eta \nu \lambda^2 \beta \|\sigma\|_0^2 \\
& \nonumber \quad + (N^2 \moll'(\Phi(r)) / 4\delta - 2 \eta \nu \lambda^2 \kmin^2  )|z|^2.
\end{align}

We estimate the first term as in the proof of Proposition \ref{prop:exist-unique} equation \ref{eq:psi-phi-ineq},  
$- (1-\delta) \moll'(\Phi(r)) |\nabla \Phi(r)|^2 \leq - (1-\delta) \gamma \moll(\Phi(r))$
for all $r \in \rbb^2$.

Regardless of the value of $r$, we require that the coefficient of
$|z|^2$ in \eqref{eq:gen-v-collected} satisfy the constraint
$N^2 \moll'(\Phi(r))/4\delta - 2 \eta \nu \lambda^2 \kmin^2 \leq - \eta \gamma (1 - \delta)$, which  is true for all $\eta$ satisfying
 \begin{equation}\label{eqn:lypcons1}
 \eta \geq \frac{N^2}{2 \nu \lambda^2 \kmin^2 - \gamma (1-\delta)}\frac{\|\moll'(\cdot)\|_\infty.}{4\delta}
\end{equation}
By choosing the $\delta$ close to 1, we can ensure that the denominator is positive.  Applying these estimates, Equation \eqref{eq:gen-v-collected} becomes
\begin{equation*}
	\gen V \leq - (1 - \delta) \gamma V + 2 \eta \nu \lambda^2 \beta \|\sigma\|_0^2.
\end{equation*}

Our final restriction involves the constant terms given in
Equation \eqref{eq:b-restriction}, with $a = (1-\delta) \gamma$ and $b = 2 \eta \nu \lambda^2 \beta \|\sigma\|_0^2$. We obtain the constraint
\begin{equation}\label{eqn:lypcons2}
\eta  \leq \frac{(1 - \delta) \gamma \psi(\phi(\rhogood)) (1 - e^{- (1-\delta) \gamma t})}{4 \nu \lambda^2 \beta \|\sigma\|_0^2}.
\end{equation}
Since $ t \geq 1$ it is enough to have
$\eta \leq \frac{(1-\delta) \gamma \psi(\phi(\rhogood)) (1 - e^{- (1-\delta) \gamma })}{4 \nu \lambda^2 \beta \|\sigma\|_0^2}$.
Combining \eqref{eqn:lypcons1} and \eqref{eqn:lypcons2}, we need to find $\eta$ such that 
\begin{equation} \label{eq:defn-eta}
	\frac{N^2 \|\moll'(\cdot)\|_\infty}{4 \delta (2 \nu \lambda^2 \kmin^2 - \gamma (1-\delta))} \leq \eta \leq \frac{(1 - \delta) \gamma \psi(\phi(\rhogood)) (1 - e^{- (1-\delta) \gamma })}{4 \nu \lambda^2 \beta \|\sigma\|_0^2}.
\end{equation}
At this point, all parameters have been fixed except for the choice of the constant $c$ in the definition of $\psi$, and the choice of $\rhogood$.  By choosing $c$ to be sufficiently small, we can diminish $\|\psi'\|_\infty$ enough that the left hand side is less than $1/4$.  Subsequently we observe that regardless of the value of $c$, $\lim_{\rho \to \rhomax} \psi(\rho) = \infty$ and so we can choose $\rhogood$ in such a way that the right-hand side is arbitrarily large.  For simplicity, we pick it so that the right-hand side is 1/2. %
\end{proof}

\section{Topological Irreducibility}
\begin{proof}[Proof of \factname \ref{control-list-one} ]

Any two points $r_0$ and $r_\ast$ in $\rgood$ can be
connected by a path consisting of two parts, $r_0 \rightarrow {| r_*|} r_0/{|r_0|}
\rightarrow \ r_*$,
a line segment (connecting $r_0$ to $|r_*| \,r_0 / |r_0|$)
and then a circular arc (connecting $|r_*| \, r_0 / |r_0|$ to
$r_*$). The length of the linear segment is less than $\rho_0$ and
the length of the circular arc will be less than $\pi \rho_0$.
Qualitatively speaking, by smoothing out the corner, there exists a
smooth curve from $r_0$ to $r_*$ with arclength less than $(1 +
\pi)\rho_0$. It follows that there exists a parametrization $\tilde
\rv$ of such a curve, and furthermore, the $\rcal$ defined by Equation \eqref{eq:R} in the statement of \factname \ref{control-list-one} in non-empty.

Given this $\tilde \rv$, we consider the linear (in $\tilde
\zv$) system
\begin{equation*}
\ddt \tilde \rv(t) = - \nabla \Phi(\tilde \rv(t)) + S(\tilde \rv(t))
\tilde \zv(t)
\end{equation*}
for every $t \in [0,1]$. There exists a unique minimal norm solution
\begin{equation*}
\tilde \zv(t) = S^\dagger(\tilde \rv(t)) \Big(\nabla \Phi(\tilde
\rv(t)) + \ddt \tilde \rv(t)\Big)
\end{equation*}
where $S^\dagger := S^* (S S^*)^{-1}$ is the Moore-Penrose
pseudoinverse  \cite{ben-israel} and $S^\ast$ is the transpose of $S$. We claim that $\tilde \zv$ is continuous and therefore bounded over
the interval $t \in [0,1]$.  Indeed, by hypothesis, both $\nabla
\Phi(\tilde \rv)$ and $\ddt \tilde \rv$ are continuous, so we
only must show that $S^\dagger(\tilde \rv(\cdot))$ is continuous.

As a finite sum of sines, $S$ is a continuous function on
$\rbb^2$. It follows that both $S^*$ and $S S^*$ are continuous as
well, and $(S S^*)^{-1}$ is continuous in any domain in which its
determinant satisfies $|\det(S(r) S^*(r))| > 0$ for all $r$ in
the domain. Because $S S^*$ is a $2 \times 2$ matrix
$$
S S^* = \left(\begin{array}{cc} |S_1|^2 & S_1 \cdot S_2 \\ S_1 \cdot
S_2 & |S_2|^2 \end{array} \right)
$$
where $S_1$ and $S_2$ are the first and second rows of $S$
respectively, the determinant simplifies to $\det(S(r)) =
|S_1(r)|^2 |S_2(r)|^2 \big(1 - \cos^2(\theta(r))\big)$ where
$\theta$ is the angle between the vectors $S_1$ and $S_2$. Noting
that $\theta$ is a continuous function of $r$ while recalling that each $S_i(r)$ is continuous and that $\rgood$ is compact, it
suffices to show that that $S_1(r)$ and $S_2(r)$ are linearly
independent for all $r \in \good_\rv$. Because the row space and
column space of a matrix have the same dimension, this reduces to
showing the column rank of $S(r)$ is two.  This follows
immediately from the hypothesis that the active mode vector set
$\modeset$ contains at least three pairwise linearly independent
vectors, which we label $\kv_1$, $\kv_2$ and $\kv_3$.  Among the
three columns $\{\sin(\lambda \kv_j \cdot r) \kvperp_j\}_{j=1}^3$ at most
one of the sine coefficients is zero, leaving at least two linearly
independent columns.

We conclude that the control $\tilde \zv(\cdot)$ is
well-defined, continuous and has a magnitude which is bounded above
by
\begin{equation*}
|\tilde \zv(t)| \leq M_1 := \sup_{r \in \rgood}\|S^\dagger(r)\|_F
(|\nabla \Phi(r)| + 5 \rhogood)
\end{equation*}
for all $t \in [0,1]$.
\end{proof}
\begin{proof}[Proof of \factname \ref{control-list-two}]

Let the constants $\delta_r \in (0,\ep_1/2)$, $T>0$ and $M_2 >0$ be given. Suppose 
$\tilde \zv \in C([0,T],\rbb^N)$ is a deterministic control with $|\zv|_\infty \leq M_2$ such that $\tilde \rv = \Psi(\tilde r_0,\tilde \zv)$ satisfies $\tilde \rv(t) \in \rgood$ for all $t \in [0,T]$. 

We will show that there exist positive constants 
$\gamma$, $\delta_{0}$, and $\delta_z$
such that if $|r_0 - \tilde r_0| \leq \delta_0$ and $\zv(\cdot) \in \mathcal{Z}(\tilde \zv, M_2, \gamma, \delta_z)$, 
then
\begin{equation} \label{eq:control-r-approx-bound}
\sup_{t \in [0, T]} |\rv(t) - \tilde \rv(t)| \leq \delta_r.
\end{equation}
To this end, define $\hv(t) := \rv(t) - \tilde \rv(t)$. Then $\hv$
satisfies the integral equation
$$
\hv(t) = \hv(0) + \int_0^t \nabla \Phi(\rv(s)) - \nabla \Phi(\tilde
\rv(s))ds + \int_0^t S(\rv(s)) \zv(s) - S(\tilde \rv(s)) \tilde
\zv(s) ds
$$
As functions of $\rv$, both $\nabla \Phi$ and $S$ are
locally Lipschitz.  Let $\rgood^+ \subset \rbb^2$ be the annulus centered at the origin with inner radius $\ep_1/2$ and outer radius $\rho_0 + \ep_1/2$.  Although the deterministic control is defined so that $\tilde \rv$ stays in $\good_r$, instances of a the actual connector process $\rv$ may wander slightly out of the good region.  It is with respect to this enlarged set that we take the local Lipschitz constants, $\lambda_\Phi > 0$ and $\lambda_S > 0$ such that for all $r, \tilde r \in \good^+$,
\begin{align*}
|\nabla \Phi(r) - \nabla \Phi(\tilde r)| &\leq \lambda_\Phi |r - \tilde r|, \quad
\|S(r) - S(\tilde r)\|_F \leq \lambda_S |r - \tilde
r|.
\end{align*}
Observing that $ |S(r)z - S(\tilde r) \tilde z| \leq \lambda_S |r - \tilde
r| |z| + \|S(\tilde r)\|_F |z - \tilde z| $ for all $r, \tilde r \in \rgood^+$ yields
\begin{align*}
|\hv(t)| &\leq |\hv(0)| + \int_0^t (\lambda_\Phi + \lambda_S
|\zv(s)|) |\hv(s)| ds + \int_0^t \|S(\tilde \rv(s))\|_F |\zv(s) - \tilde \zv(s)| ds.
\end{align*}
By virtue of the assumption that $\zv \in \mathcal{Z}(\tilde \zv, M_2, \gamma, \delta_z)$, defined in \eqref{eqn:fancyz2} the second integral satisfies the bound
\begin{align*}
\int_0^t \|S(\tilde \rv(s))\|_F |\zv(s) - \tilde \zv(s)| ds &\leq \sup_{r \in \rgood} \|S(r)\|_F \int_0^t M_2 e^{-\gamma s} + \delta_z  ds,
\end{align*}
and so after simplifying we have 
$|\hv(t)| \leq \int_0^t \beta |\hv(s)| ds + g(t)$
where  $\beta = \lambda_\Phi + (2 M_2 + \delta_z) \lambda_S$ and $g(t) = \delta_0 +
\sup_{r \in \good} \|S(r)\|_F \big(\frac{M_2}{\gamma} + \delta_z
t\big)$. Using the integral form of Gronwall's Inequality yields
$
|\hv(t)| \leq g(t) + \int_0^t g(s) \beta e^{\beta (t-s)}  ds
$.
After substituting in the values of $g$ and $\beta$ and integrating, we see that for all $t \in [0,T]$,
$$
|\hv(t)| \leq \Big[\delta_0 + \sup_{r \in \rgood} \|S(r)\|_F
\Big(\frac{M_2}{\gamma} + \frac{\delta_z}{\lambda_\Phi + M_2 \lambda_S
}\Big)\Big] e^{(\lambda_\Phi + M_2 \lambda_S) T}
$$
Taking $\delta_0$ and $\delta_z$ sufficiently small while taking $\gamma$ sufficiently large yields \eqref{eq:control-r-approx-bound}.
\end{proof}

\begin{proof}[Proof of \factname \ref{control-list-three}]
	Let the constants $\gamma > 0$, $\delta_z > 0$ and $M > 0$ be given, along with $\tilde Z \in C([0,T];\rbb^N)$ satisfying $|\tilde Z|_\infty < M$. As in the proof of Lemma \ref{lem:nonlinear-near-zero} the noise vector $Z(t) = (Z_1(t), Z_2(t), \ldots, Z_N(t))$ can be written
	$\zv(t) = e^{-\Lambdav t} z_0 + \int_0^t e^{-\Lambdav(t - s)}
        B d\Wv(s)$
where $\Lambdav$ is a diagonal matrix whose entries $\{\lambda_\kv\}_{\kv \in \modeset}$ are given by $\lambda_\kv := \lambda^2 \nu |\kv|^2$ and $B$ is a diagonal matrix whose entries $\{b_k\}_{k\in \modeset}$ are given by $\sqrt{2 \beta \nu} \length \sigma_\kv$. 
	
Again view the stochastic integral as a time change of a Brownian motion. As before $M_\kv(t) := \int_{0}^t e^{\lambda_{\kv} s} b_k d W_{\kv}(s)$ is a continuous martingale with quadratic variation $\langle M_\kv, M_\kv \rangle_t = b_k^2 (e^{2 \lambda_\kv t} - 1)/ 2 \lambda_{\kv}$, we observe that for any $t > 0$, $M_\kv(t)$ has the same distribution as $\tilde W( \langle M_\kv, M_\kv \rangle_t )$ where $\tilde W$ is a standard Brownian motion. For any continuous curve $\Gamma$ with $\Gamma(0) = 0$, $\tilde T > 0$ and $\delta > 0$
\begin{equation*}
	\mathbb{P}\Big\{\sup_{t \in [0,\tilde T]} |\tilde W(t) - \Gamma(t)| \leq \delta\Big\} > \tilde p
\end{equation*}
for some $\tilde p > 0$ (see \cite{durrett} for example). Since we have assumed there are only a finite number of active modes, and the modes are independent, \factname \ref{control-list-three} follows immediately from the union bound.
\end{proof}

\bibliography{bead-spring}
\bibliographystyle{alpha}

\end{document}